\documentclass{amsart}

\usepackage{amsmath}
\usepackage{amsthm,amssymb,color,comment}
\usepackage{bbm}
\usepackage{hyperref}
\usepackage[TS1,T1]{fontenc}
\usepackage{inputenc}
\usepackage{dsfont}
\usepackage{tikz}
\usepackage{enumitem}
\usepackage{soul}
\usepackage[sort]{cite}% to arrange citations in increasing order

\numberwithin{equation}{section}

\newtheorem{thm}{Theorem}[section]

\newtheorem{lem}[thm]{Lemma}
\newtheorem{cor}[thm]{Corollary}

\newtheorem{Def}[thm]{Definition}

\theoremstyle{definition}
\newtheorem{Ass}[thm]{Assumption}
\newtheorem{rem}[thm]{Remark}

\newtheorem{Not}[thm]{Notation}

\DeclareMathOperator{\DIV}{div}

\newcommand{\R}{\mathbb{R}}
\newcommand{\N}{\mathbb{N}}
\newcommand{\p}{\partial}

\newcommand{\supp}{\text{supp}}
\newcommand{\diff}{\mathop{}\!\mathrm{d}}

\newcommand{\symm}{\mathbb{R}^{d\times d}_{\text{sym}}}

\newcommand{\wstar}{\overset{\ast}{\rightharpoonup}}

\makeatletter
\newcommand{\doublewidetilde}[1]{{%
  \mathpalette\double@widetilde{#1}%
}}
\newcommand{\double@widetilde}[2]{%
  \sbox\z@{$\m@th#1\widetilde{#2}$}%
  \ht\z@=.9\ht\z@
  \widetilde{\box\z@}%
}
\makeatother
\author{Jakub Woźnicki}
\address{Institute of Mathematics of Polish Academy of Sciences, Jana i J\k edrzeja \'Sniadeckich 8, 00-656 Warsaw, Poland; Faculty of Mathematics, Informatics and Mechanics, University of Warsaw, Stefana Banacha 2, 02-097 Warsaw, Poland}
\email{jw.woznicki@student.uw.edu.pl}
\thanks{Jakub Woźnicki was supported by National Science Center, Poland through project no. 2021/43/O/ST1/03031}

\begin{document}

\title[Vlasov-type/non-Newtonian system]{Vlasov equation coupled with non-Newtonian fluids with discontinuous-in-time stress tensor}

\begin{abstract}
We analyze the system of equations describing the flow of a dilute particle system coupled with an incompressible non-Newtonian fluid in a bounded domain. In this setting, both PDEs are connected via a drag force, or the friction force. We are interested in a thin spray regime, which means that we neglect the inter-particle interactions. When it comes to the fluid system, the Cauchy stress tensor is supposed to be a monotone mapping and has asymptotically $(s-1)$-growth with the parameter $s$ depending on the spatial and time variable. We do not assume any smoothness of $s$ with respect to time variable and assume the log-H\"{o}lder continuity with respect to spatial variable. An example of materials, which satisfy those assumptions, are those whose properties are instantaneous, e.g. changed by the switched electric field. We show the long time and the large data existence of weak solution provided that $s\ge\frac{3d+2}{d+2}$.
\end{abstract}


\maketitle

\section{Introduction}
 We consider the system of partial differential equations
\begin{align}\label{sys:main_system}
    \left\{\begin{array}{ll}
    \partial_t f(t, x, v) + v\cdot\nabla_x f(t, x, v) + \DIV_v((u(t, x) - v)f(t, x, v)) = 0,\\
    \partial_t u(t, x) + \DIV_x (u(t,x)\otimes u(t, x)) + \nabla_x p(t, x) = \DIV_x S(t, x, Du) - \int_{\R^d}(u(t, x) - v)f(t, x, v)\diff v,\\
    \DIV_x u(t, x) = 0,
    \end{array}\right.
\end{align}
describing a kinetic, Vlasov-type equation coupled together with an incompressible, homogeneous, non-Newtonian fluid. Here, $f(t, x, v)$ is the particle distribution function in $(x, v)\in \Omega\times \R^d$, $u$ is the velocity of the fluid, $p$ denotes pressure, $S$ is the constitutively determined part of the Cauchy stress tensor depending on the symmetric gradient $Du$ and for simplicity, we consider that the density of the fluid is equal to one. We formulate our problem \eqref{sys:main_system} on a product $(0, T)\times\Omega\times\R^d$, where $\Omega \subset \mathbb{R}^d$ is a Lipschitz domain and $d$ denotes the dimension. The above system is completed by the specular reflection, and the no-slip boundary conditions, i.e.:
$$
u \equiv 0\text{, on }\partial\Omega,\qquad f(t, x, v) = f(t, x, v^*),\text{ for }x\in\partial\Omega,\text{ and }v\cdot n(x) < 0,
$$ 
where $v^* = v - 2(v\cdot n(x))n(x)$ is a specular reflection of the kinetic velocity $v$ and $n(x)$ is an outer normal vector of $\Omega$. We set the initial conditions as $f_0(x, v)$, $u_0(x)$. Building on the theory developed in \cite{bulicek2023nonnewtonian}, where a similar problem is considered, but without a kinetic component, our goal will be to prove the existence of weak solutions to the system \eqref{sys:main_system}. \\

 When it comes to the fluid-kinetic models, their applications are vast and range from modeling of reaction flows of sprays, aerosol therapy, atmospheric pollution, chemical engineering and go even to waste water treatment, dusk collection, and others, see \cite{caflish1983dynamic, orourke1981collectivedrop, ranz1952evaporation, ranz1952evaporation2, soo1990multiphase, williams1985combustion, desvillettes2010someaspects, dufour2005modelisation, baranger2005amodelling, Gemci2002anumerical, Boudin2015modelling}. In the present paper, we are interested in the so-called thin spray regime, which in opposition to the thick spray regime ignores the inter-particle dynamics, such as collisions or break-ups, hence the coupling between the kinetic description of the particle movement and the fluids equations is done through the drag force. For an instance of thick spray regime considerations we refer to \cite{boudin2003amodeling}, and for further citations and discussion to \cite{hamdache1998global}.\\ 

 The research on the fluid-kinetic models has been done for decades. Looking at compressible models, the existence of global weak solutions for Navier--Stokes coupled with Vlasov--Fokker--Planck equations in a bounded domain (with the same boundary conditions as the ones presented in this work) has been proven in \cite{mellet2007globalweak}. The proof existence of local-in-time classical solutions on the whole space for Euler--Vlasov equations has been carried out by Baranger and Desvillettes in \cite{baranger2006coupling}, and the similar results have been obtained by Mathiaud \cite{mathiaud2010local} for Euler--Vlasov--Blotzmann system. More recently Choi and Jung \cite{choi2022onregular} showed uniqueness and local-in-time existence of strong solutions for Vlasov--Navier--Stokes with degenerate viscosities and vacuum. Worth noting are also results \cite{chae2013global, choi2015global} where the authors investigated global existence of classical solutions close to the equilibrium in the periodic and the whole space case. \\

 Even more contributions have been made in an incompressible case, going back as early as in 1998 Hamdache \cite{hamdache1998global} showed the existence of global weak solutions for Vlasov--Stokes system in a bounded domain (for the same boundary conditions as in this work). This result was later extended by Yu \cite{yu2013global} to Vlasov--Navier--Stokes. In a periodic case Boudin et. al. \cite{boudin2009global} have proven the existence of global weak solutions for the same system, which was later extended in \cite{mucha2018flocking} to Vlasov/non-Newtonian system (including an alignment force), where the stress tensor has $L^{p-1}$ growth, $p\geq \frac{11}{5}$ and is generated by a potential, and most recently further generalized to general $p-1$ growth potentials $p > \frac{8}{5}$ in \cite{choi2025coupledvlasovnonnewtonianfluid}, but without the alignment force present. Further results on a Newtonian case, but with a time-dependent domain, and absorption boundary conditions for the kinetic component include \cite{boudin2017global, boudin2020global}, where the authors show the global existence of weak solutions for Vlasov--Navier--Stokes, and Vlasov--Navier--Stokes coupled with convection-diffusion equations respectively. We would like to also mention the hydrodynamical limits for light and fine particle regimes which were done in \cite{goudon2004hydrodynamic, goudon2004hydrodynamic2}.\\

 At last, worth noting are techniques used to treat the difficult kinetic terms that appear in the coupling between the two interacting systems. The most important of which is the velocity averaging lemma, which allows us to deduce compactness of the kinetic density $\int_{\R^d}f\diff v$ and the kinetic momentum $\int_{\R^d}v\,f\diff v$, from quite weak information on $f$ itself. We refer the reader to \cite{perthrame1998alimiting} and \cite{rein2004globalweak}. For historical reasons we also cite the preceeding works on less complicated kinetic systems \cite{agoshkov1984spaces, diperna1989onthecauchy, diperna1989global, diperna1991lpregularity, golse1988regularity}.\\

 Let us speak a while about a motivation of tackling the specific non-Newtonian model at hand. In most of the cases (see e.g. \cite{choi2025coupledvlasovnonnewtonianfluid, mucha2018flocking}), for models of non-Newtonian fluids, one assumes the Cauchy stress tensor to be of power type
\begin{align}\label{stress_power_type}
S(t,x,Du) \sim (\nu_0 + \nu_1\,|Du|^{p-2}) \, Du.
\end{align}
The first existence results concerning the system describing the non-Newtonian fluids (without the kinetic opponent) were proved for $p \geq \frac{11}{5}$ (in 3D) by Lions and Ladyzhenskaya in \cite{MR0259693, MR0254401}. After that, there have been many generalizations, from the higher regularity method in \cite{malek1996weakandmeasure} giving the bound $p\ge \frac{9}{5}$, followed by the $L^{\infty}$-truncation method for $p\ge \frac{8}{5}$, see \cite{MR1713880}, the Lipschitz truncation method for $p>\frac{6}{5}$, see \cite{MR2668872,MR2001659,MR3023393}, up to a new definition of a solution in \cite{MR4102807} leading to the theory for all $p>1$.  We refer the reader to the extensive review in \cite{MR4076814} in context of fluids with very complicated rheology.

 Our area of interest are the forms of the stress tensor $S$ analogous to \eqref{stress_power_type}, that is
\begin{equation}\label{eq:stress_tensor_var_exp}
S(t,x,Du) \sim (\nu_0 + \nu_1\,|Du|^{s(t,x)-2}) \, Du,
\end{equation}
where the exponent $s(t,x)$ depends on the time variable $t$ and the spatial variable $x$, which can be found in the literature under the name non-standard growth. The form \eqref{eq:stress_tensor_var_exp} appears in the analysis of the electrorheological fluids whose mechanical properties dramatically change when an external electric field is applied, see~\cite{RAJAGOPAL1996401}. To see the studies from the mathematical point of view, we recommend looking at~\cite{diening2011lebesgue,MR1930392, MR1810360, MR2037246, MR3474486, MR2346460, MR2610563, bulicek2023nonnewtonian}. Further generalizations to the micropolar fluids and also the chemically reacting fluids are available in \cite{MR3474486, MR3373570, MR3904139, MR3834434, ko2018existence}.\\

 The power-type structure such as the one in \eqref{eq:stress_tensor_var_exp} can be generalized as an inequality, which involves the growth and the coercivity of the stress tensor
\begin{equation}\label{eq:ass_in_intro}
c \, S(t, x, \xi) : \xi \geq | \xi|^{s(t, x)} + | S(t, x, \xi)|^{s'(t, x)} - h(t, x),
\end{equation}
where $c$ is some constant, $h \in L^{1}(\Omega_T)$ and $s'(t,x)$ is the H{\"o}lder conjugate exponent  to $s(t,x)$, i.e., $s':=\frac{s}{s-1}$. Analysis of \eqref{sys:main_system} with the Cauchy stress tensor of the form \eqref{eq:ass_in_intro} requires the concept of generalized Lebesgue spaces $L^{s(t,x)}$. We would like to mention as well, that one could generalize \eqref{eq:ass_in_intro} by replacing power-type function with a generalized $N$-function. This would require the application of the general Musielak--Orlicz spaces \cite{chlebicka2019book} as in \cite{MR3007700,MR3190317,MR2466805, MR2597212}. We remark that all results of this paper can be formulated in this setting but we decided not to do so for the sake of clarity. \\

 In the current work, we show the existence of global-in-time and large-data solutions to \eqref{sys:main_system} for exponents $s(t,x)$ being discontinuous in the time variable. The previous approaches were based on the so-called log-H{\"o}lder continuity of the exponent $s(t,x)$, which allows one to use the density of smooth functions in the space $L^{s(t,x)}$,  see \cite{MR3847479, cruz2013variable}. In fact, the log-H{\"o}lder continuity is necessary for the density to be true, see \cite[Example 6.12]{cruz2013variable}. Nevertheless, inspired by \cite{bulicek2023nonnewtonian}, where the problem analogous to \eqref{sys:main_system} was considered, but without the kinetic component, and with $s(t,x)$ being discontinuous in time and log-H{\"o}lder continuous in space (extending the results of \cite{chlebicka2019parabolic}), we do not require the smoothness of $s$ with respect to the time variable here. Such a case is not only interesting from a mathematical point of view, but appears naturally in physics, see for example the model in \cite{ruzicka2000electrorheological}. There, one can find a mathematical description of the behavior of the fluid composed of charged particles moving in the electric field ${E}(t,x)$, where ${E}(t,x)$ solves certain Maxwell equation. In this case, the exponent $s$ can be assumed to be a smooth function of $|{E}|^2$ cf. \cite[Chapter 4, eq. (1.1)--(1.7)]{ruzicka2000electrorheological} so that $s$ depends on $t$ and $x$. Furthermore, the regularity of $s$ with respect to $t$ and $x$ is the same as the regularity of $E(t,x)$. This way, the regularity of $s$ with respect to the space is inherited from the regularity of the field $E$. On the other hand, switching the electric field may cause "jumps" and therefore its irregularity in time.\\

 For further discussion on mathematics of the model \eqref{sys:main_system}, without the kinetic component we refer to \cite{bulicek2023nonnewtonian}, where one can find many detailed explanation on particular assumptions, and in fact, is the main inspiration for tackling this problem. In the mentioned paper, one can also find the technical results, which will be heavily exploited throughout current work.\\

 The structure of the paper is as follows. In Section~\ref{section:2}, we introduce the precise assumption on the Cauchy stress tensor $S$, on the variable exponent $s(t,x)$ and we formulate the main result of the paper. Then, we recall the results from \cite{bulicek2023nonnewtonian} regarding an abstract parabolic equation in Section~\ref{section:4} and show the validity of the local energy equality. Afterwards, in Section~\ref{section:5} we introduce a classical approximation problem for which the theory is available and in Section~\ref{section:6} we finally pass to the limit and finish the proof of the main theorem. The classical but auxiliary tools are recalled for the sake of clarity in Appendix.

\section{Preliminaries and the main result}\label{section:2}
 Let us start with the notation in the paper. By $d = 2, 3$, we denote the dimension of the space, $\Omega \subset \mathbb{R}^d$ is a Lipschitz spatial domain, $T > 0$ is the length of time interest, and $\R^d$ is the kinetic velocity vector space. We write $x$ for an element of $\Omega$, $t$ for an element of $(0,T)$, and $v$ for an element of $\R^d$. The corresponding parabolic domain will be denoted with $\Omega_T:= (0,T) \times \Omega$, and its specific subdomains as $\Omega_t = (0,t) \times \Omega$. Moving forward, for any $a,b\in \mathbb{R}^d$ we write $a\cdot b$ for the standard scalar product $\sum_{i = 1}^da_i b_i$. Similarly, the space $\symm$ denotes the space of symmetric $d\times d$ matrices and for any $A,B \in \symm$ we denote the scalar product $\mathrm{tr}(A^TB)$ by $A:B$. Next, the symbol $\otimes$ is reserved for the tensorial product, i.e., for $a,b\in \mathbb{R}^d$ we denote $a\otimes b \in \symm$ as $(a\otimes b)_{ij}:=a_ib_j$ for $i,j=1,\ldots, d$. Throughout the work, we will use the standard notation for Sobolev and Lebesgue function space and frequently do not distinguish between scalar-, vector- or matrix-valued functions. To shorten the notation, we frequently use the following simplifications. When $f \in L^p(\Omega)$, we simplify it to $f \in L^p_x$. Similarly, if $f \in L^p(0,T; L^q(\Omega))$, $f \in L^p(0,T; W^{1,q}(\Omega))$ or $f \in L^p(0,T; W^{1,q}_0(\Omega))$, then we write $f \in L^p_t L^q_x$, $f \in L^p_t W^{1,q}_x$ or $f \in L^p_t W^{1,q}_{0,x}$ respectively (here, $W^{1,q}(\Omega)$ and $W^{1, q}_0(\Omega)$ are the usual Sobolev spaces). Furthermore, whenever $f\in L^p(0, T; L^q(\Omega\times\R^d)$, we write $f\in L^p_t L^q_{x,v}$, and so on. Regarding any exponent $p \in [1,\infty]$, we denote by $p'$ its H{\"o}lder conjugate defined by the equation $\frac{1}{p} + \frac{1}{p'} = 1$. Symbol $\nabla_x u$ is used for the spatial gradient of $u$ and $Du$ denotes its symmetric part, i.e. $Du = \left(\nabla_x u + (\nabla_x u)^T\right)/2$. We mostly work in the variable exponent spaces $L^{s(t,x)}(\Omega_T)$, and for the more detailed discussion we refer to Appendix \ref{app:musielaki}. The last special function space related to the fluid mechanics is $ L^2_{0,\DIV}(\Omega)$, which is defined as a closure of the set $\{u\in C_c(\Omega; \mathbb{R}^d), \, \DIV u=0\}$ in $L^2(\Omega)$. The constant $C > 0$ refers to a universal constant, which may be different even in single considerations, but depends only on the data.\\

 As a last note, we would like to underline, that the set-up below for the fluid system is completely analogous to the one presented in \cite{bulicek2023nonnewtonian}.

\subsection*{Assumptions on data}
Let us now state the needed assumptions on the exponent function $s(t, x)$:
\begin{Ass}\label{ass:exponent_cont_space}
We assume that a measurable function $s(t,x): \Omega_T \to [1,\infty)$ satisfies the following:
\begin{enumerate}[label=(A\arabic*)]    \item\label{ass:cont} (continuity in space) $s(t,x)$ is a log-H\"older continuous functions on $\Omega$ uniformly in time, i.e. there is a constant $C$ such that for all $x, y \in \Omega$ and all $t \in [0,T]$
$$
|s(t,x) - s(t,y)| \leq -\frac{C}{\log|x-y|},
$$
\item\label{ass:usual_bounds_exp} (bounds) it holds that $ \frac{3d + 2}{d+2} =: s_{\text{min}}\leq s(t,x)  \leq s_{\text{max}} < +\infty$ for a.e. $(t,x) \in \Omega_T$.
\end{enumerate}
\end{Ass}
 For later purposes we also define an exponent $s_0$ as
\begin{equation} \label{s0}
s_0 := 3 + \frac{2}{d}.
\end{equation}
As mentioned in the introduction, the condition \ref{ass:cont} guarantees good approximation properties (with respect to spatial variable $x$) in the variable exponent space. As for the Assumption \ref{ass:usual_bounds_exp}, it is completely standard and appears in the classical Navier--Stokes considerations. It is connected to the  continuity of the three-linear form
$$
v\longmapsto \int_{\Omega_T} v\otimes v : \nabla v\diff x\diff t,
$$
which is used to obtain the energy equality.\\

 Next, we look at two technical lemmas regarding a possible decomposition of the spatial space $\Omega$. They are taken from \cite{bulicek2023nonnewtonian}, but since proofs are brief, we provide them for the sake of completeness.

\begin{lem}\label{lem:decomposition_Omega}
There exists $r>0$ and an open finite covering $\{\mathcal{B}^i_r\}_{i=1}^N$ of $\Omega$ by balls of radii $r$ such that if we define
$$
q_i(t) := \inf_{x \in \mathcal{B}_{2r}^i} s(t,x), \qquad r_i(t) := \sup_{x \in \mathcal{B}_{2r}^i} s(t,x), \qquad R_i(t) := q_i(t)\,\left(1+\frac{2}{d}\right)
$$
we have for all $i=1,\ldots, N$
$$
s_{\text{min}} \leq q_i(t) \leq s(t,x) \leq r_i(t) < R_i(t) %{\color{red}\,\leq s_{\text{max}}}
\qquad \mbox{ on } (0,T) \times (\mathcal{B}^{i}_{2r} \cap \Omega)
$$
and
$$
R_i(t) - r_i(t) \geq \frac{s_{\text{min}}}{d}.
$$
\end{lem}
\begin{proof}%[Proof of Lemma \ref{lem:decomposition_Omega}]
We cover $\Omega$ with balls $\mathcal{B}_r^i$ of equal radius $r$ and the only problem is to find radius $r$ satisfying assertions of the lemma. By \ref{ass:cont} in Assumption \ref{ass:exponent_cont_space} we can choose $r$ such that
$$
\sup_{x \in \mathcal{B}_{2r}^i \cap \Omega} s(t,x) - \inf_{x \in \mathcal{B}_{2r}^i \cap \Omega} s(t,x) \leq \frac{s_{\text{min}}}{d}.
$$
Since $s_{\text{min}} \leq q_i(t)$, the conclusion follows.
\end{proof}
\begin{Not}\label{not:zeta}
In what follows, we always consider the covering constructed in Lemma \ref{lem:decomposition_Omega}. We also write $\zeta_i$, $i = 1,\ldots, N$ for the partition of unity related to the open covering $\{\mathcal{B}^{i}_{r}\}$ of $\Omega$, that is $\mbox{supp}\,\zeta_i \subset \mathcal{B}^{i}_{r}$ and $\sum_{i=1}^N \zeta_i(x) = 1$ for all $x\in \Omega$.
\end{Not}

 Now, we observe that the fact $Du \in L^{s(t,x)}(\Omega_T)$ implies certain regularity properties.
\begin{lem}\label{lem:simple_integrability_lemma}
Suppose that Assumption \ref{ass:exponent_cont_space} holds true. Let $u$ be such that $Du \in L^{s(t,x)}(\Omega_T)$, $u \in L^1_t W^{1,1}_{0,x}$ and  $u\in L^{\infty}_t L^2_x$. Then, $u \in L^{q_i(t)}(0,T; W^{1,q_i(t)}(\mathcal{B}_{2r}^i)) \cap L^{s_{\text{min}}}(0,T;W^{1,s_{\text{min}}}_0(\Omega))$ for all $i=1,\ldots, N$. The norm of $u$ in these spaces depends only on
$
\|Du\|_{L^{s(t,x)}}, \|u\|_{L^{\infty}_t L^2_x}.
$
\end{lem}
\begin{proof}
As $q_i(t) \leq s(t,x)$ on $\mathcal{B}_{2r}^i$, we deduce $Du \in L^{q_i(t)}(0,T; L^{q_i(t)}(\mathcal{B}_{2r}^i))$. Then, the generalized Korn inequality implies that (for fix $t$ and $i$)
$$
\|\nabla u\|_{L^{q_i(t)}(\mathcal{B}_{2r}^i)}\le C(r, s_{\text{min}}, s_{\text{max}})(\|D u\|_{L^{q_i(t)}(\mathcal{B}_{2r}^i)} + \|u\|_2).
$$
Then rasing the inequality to the $q_i(t)$ power, integrating over $t\in (0,T)$ and using the assumptions on $u$ we have the first part of the statement. The second statement can proved exactly in the same way using that $s_{\text{min}} \leq s(t,x)$ on $\Omega_T$ and the standard Korn inequality for functions having zero trace.
\end{proof}

\subsection*{Assumptions on the stress tensor}
Concerning the stress tensor $S: (0, T) \times \Omega\times \symm \to \symm$ we assume the following:
\begin{Ass}\label{ass:stress_tensor} We assume that
\begin{enumerate}[label=(S\arabic*)]
\item \label{T1}$S(t,x, \xi)$ is a Carath\'{e}odory function and $S(t, x, 0) = 0$,
\item \label{coercitivity_stress_tensor} (coercivity and growth conditions) There exists a positive constant $c$  and a non-negative, integrable function $h(t, x)$, such that for any $\xi\in\symm$ and almost every $(t, x) \in \Omega_T$
$$
c \, S(t, x, \xi) : \xi \geq | \xi|^{s(t, x)} + | S(t, x, \xi)|^{s'(t, x)} - h(t, x)
$$
\item \label{monotonicity_stress_tensor}(monotonicity) S is monotone, i.e.:
$$
(S(t, x, \xi_1) - S(t, x, \xi_2)) : (\xi_1 - \xi_2) \geq 0
$$
for almost every $(t, x) \in \Omega_T$.
\end{enumerate}
\end{Ass}

\subsection*{Main result}
Having introduced all the needed notation, we may finally state the main theorem.
\begin{thm}\label{thm:the_main_result}
Let $S(t, x, \xi)$ satisfy the Assumption~\ref{ass:stress_tensor} with the exponent $s(t, x)$ satisfying Assumption~\ref{ass:exponent_cont_space}. Then, for any $f_0\in L^\infty_{x,v}$, $(1 + |v|^2)f_0\in L^1_{x,v}$, and any $u_0\in L^2_{0,\DIV}(\Omega)$, there exists $f\in L^\infty_t (L^1\cap L^\infty)_{x,v}$, $f\geq 0$ and $u\in L^\infty_t L^2_x \cap L^{s_{\text{min}}}_t W^{1, s_{\text{min}}}_{0, x}$, $Du\in L^{s(t, x)}(\Omega_T)$, such that $\DIV u=0$ almost everywhere in $\Omega_T$, and the following equations are satisfied
\begin{align}\label{eq:the_main_result_kinetic}
    \int_{0}^T\int_{\Omega\times\R^d}f(\p_t\phi + v\cdot\nabla_x\phi + (u- v)\cdot\nabla_v\phi)\diff v\diff x\diff t = \int_{\Omega\times\R^d}f_0\,\phi(0, x, v)\diff v\diff x
\end{align}
for any $\phi\in C^\infty_c([0, T)\times \overline{\Omega}\times\R^d)$, as well as
\begin{equation}\label{eq:the_main_result}
\begin{split}
    \int_{\Omega_T}-u  \cdot \partial_t \phi &- u\otimes u : \nabla_x\phi + S(t, x, Du):D\phi\diff x\diff t\\
    &= -\int_{\Omega_T\times\R^d}(u - v)f \cdot \phi\diff v\diff x\diff t + \int_{\Omega}u_0\cdot \phi(0,x)\diff x
\end{split}
\end{equation}
for any $\phi\in C^\infty_c([0, T) \times \Omega)$ fulfilling $\DIV \phi = 0$ almost everywhere in $\Omega_T$.
\end{thm}

\section{Local energy equality}\label{section:4}

 In the present section, we recall and combine several results from \cite{bulicek2023nonnewtonian}, which are applied to a general abstract identity
\begin{align}\label{eq:gen_identity_def}
    \int_{\Omega_T}-u \cdot \p_t \phi - u\otimes u:\nabla_x\phi + (\alpha +\theta\, \beta) : D\phi  \diff x \diff t = \int_{\Omega}u_0(x) \cdot\phi(0, x) \diff x + \int_{\Omega_T}F  \cdot\phi \diff x \diff t
\end{align}
which is required to be satisfied for any vector-valued $\phi \in C^\infty_c([0, T) \times \Omega)$ fulfilling $\DIV \phi = 0$ in $\Omega_T$. Here, $\alpha, \beta: \Omega_T \to \symm$, $\theta \geq 0$ and the term $\theta \, \beta$ can be seen as a regularizing term. Whenever $\theta = 0$ we assume
\begin{equation}\label{eq:regularity_estimates_u}
u \in L^\infty_t L^2_x \cap L^{R_i(t)}((0,T)\times \mathcal{B}^i_{2r}) \cap L^{s_0}_{t,x}, \qquad Du \in L^{s(t,x)}(\Omega_T)
\end{equation}
\begin{equation}\label{eq:regularity_estimates_not_u_theta_0}
\alpha \in L^{s'(t,x)}(\Omega_T) \mbox{ is a symmetric matrix}, \qquad F\in L^{2}_{t,x},
\end{equation}
where, we recall \eqref{s0}, i.e., that  $s_0 = 3 + \frac{2}{d}$. Whereas for $\theta > 0$, we have a stronger assumption
\begin{equation}\label{eq:regularity_estimates_not_u_2}
\beta \in L^{s_{\max}'}(\Omega_T) \mbox{ is a symmetric matrix}, \quad Du \in L^{s_{\max}}(\Omega_T), \qquad \theta > 0,
\end{equation}
which shall disappear for $\theta \to 0^+$. The uniform estimates with respect to $\theta \geq 0$, will only be obtained only in terms of~\eqref{eq:regularity_estimates_u}--\eqref{eq:regularity_estimates_not_u_theta_0}. We want to note here, that $s_0\le R_i(t)$, which follows from the definition of $R_i(t)$ in Lemma~\ref{lem:decomposition_Omega}, the definition of $s_0$ in \eqref{s0} and the assumption $s_{\text{min}}\ge \frac{3d+2}{d+2}$. \\

 To properly state the theorems below, we need to define the pressures in our system. Let us extend all functions to $\R^d$ with zero and denote by $p_1^i$, $p_2^i$, $p_3$, $p_4$ as the unique functions satisfying for almost all $t\in (0,T)$
\begin{align}
 -\Delta p_{1}^i &= \DIV \DIV (\alpha \, \zeta_i) \mbox{ in } {\R^d}, \quad   &&{p_1^i(t,\cdot) \in L^{r_i'(t)}(\R^d)}, \label{eq:defp1A}\\
 \Delta p_{2}^i &= \DIV \DIV (u\otimes u \, \zeta_i) \mbox{ in } \R^d, \quad   &&{p_2^i(t,\cdot) \in L^{R_i(t)/2}(\R^d)}, \label{eq:defp2A}\\
 -\Delta p_3 &= \DIV F \mbox{ in } \R^d, \quad  &&p_3(t,\cdot) \in W^{1,2}_{\textrm{loc}}(\R^d), \label{eq:defp3} \\
 -\Delta p_4 &= \DIV \DIV (\theta\, \beta) \mbox{ in } {\R^d}, \quad  &&{p_4(t,\cdot) \in L^{s'_{\max}}(\R^d)}. \label{eq:defp4}
\end{align}

 Then, we obtain the following.

\begin{lem}[Lemma 4.2, \cite{bulicek2023nonnewtonian}]\label{lem:integrability_of_pi}
There exists uniquely determined $p_1^i$, $p_2^i$, $p_3$, $p_4$ satisfying \eqref{eq:defp1A}--\eqref{eq:defp4}. Moreover, there exists a constant depending only on $T$ and ${\Omega}$ such that
\begin{align*}
&\|p_{1}^i\|_{L^{r_i^{'}(t)}_{t,x}} \leq C \|\alpha \, \zeta_i\|_{L^{r_i^{'}(t)}_{t,x}},    &&\|p_{2}^i\|_{L^{R_i(t)/2}_{t,x}} \leq C \|u \, \sqrt{\zeta_i}\|_{L^{R_i(t)}_{t,x}}^2,\\
&\|p_{1}^i\|_{L^{{s'_{\max}}}_{t,x}} \leq C\, \| \alpha \|_{L^{{s'_{\max}}}_{t,x}},
&&\|p_{2}^i\|_{L^{{s_{0}}/2}_{t,x}} \leq C\, \| u \|^2_{L^{{s_{0}}}_{t,x}}\\
&\|p_3\|_{L^2_t W^{1,2}_x} \leq C\, \|F\|_{L^2_{t,x}},  &&\|p_4\|_{L^{{s'_{\max}}}_{t,x}} \leq C\, \| \theta\,\beta\|_{L^{{s'_{\max}}}_{t,x}}
\end{align*}
Moreover, for each bounded $\Omega' \Subset \R^d \setminus \mathcal{B}^i_{r}$ we have
$$
\| p_{1}^i \|_{L^{s'_{\max}}_t(0,T; C^{k}_x(\Omega'))} \leq C\left(T,k,\Omega'\right) \| \alpha \|_{L^{{s'_{\max}}}_{t,x}},
$$
$$
\| p_{2}^i \|_{L^{s_{0}/2}_t(0,T; C^{k}_x(\Omega'))} \leq C\left(T,k,\Omega'\right) \| u \|^2_{L^{{s_{0}}}_{t,x}}.
$$
\end{lem}

 At last, we can look at the main result of this section reads, concerning the extension of \eqref{eq:gen_identity_def}, and the local energy equality.

\begin{thm}[Theorem 4.1, \cite{bulicek2023nonnewtonian}]\label{thm:local_energy_equality}
Let $u$ be a solution of \eqref{eq:gen_identity_def}. Then, with $p_i$ as above,
there exists a uniquely determined function $p_h \in L^{\infty}(0,T; L^{s'_{\max}}(\Omega))$ (up to condition $\int_{\Omega} p_h(t,x) \diff x  = 0$ for almost all $t\in (0,T)$), such that for all $\varphi \in C^1_c([0,T)\times \Omega)$
\begin{equation}\label{eq:distributional_id_with_ph}
\begin{split}
-\int_{\Omega_T} &(u+\nabla p_h)\cdot \partial_t\varphi +(- u \otimes u + \alpha + \theta \beta) : \nabla \varphi \diff x \diff t-\int_{\Omega} u_0(x)\cdot \varphi(0,x)  \diff x \\
&= \int_{\Omega_T} F\, \cdot \varphi  - \sum_{i=1}^N (p_1^i + p_2^i) \,\DIV \varphi -  (p_3 + p_4) \, \DIV \varphi \diff x \diff t
\end{split}
\end{equation}
Moreover, $p_h$ is harmonic and so, it is locally smooth in the spatial variable, i.e., it satisfies
\begin{equation}\label{eq:pressure_is_harmonic}
\Delta p_h = 0 \mbox{ in } \Omega \mbox{ for a.e. } t \in [0,T].
\end{equation}
In addition, we have the following estimate valid for all $\Omega' \Subset \Omega$
\begin{equation}\label{eq:thm_estimate_on_p_Hol}
\begin{split}
\|p_h\|_{L^{\infty}_t L^{s'_{\max}}_x}+ \|p_h\|_{L^{\infty}(0,T;C^k(\Omega'))} \leq C,
\end{split}
\end{equation}
where the constant $C$ depends on $k$, $\Omega'$, $T$ and the norms $\|\alpha\|_{L^{s'(t,x)}_{t,x}}$, $\|Du\|_{L^{s(t,x)}_{t,x}}$, $\|\theta \, \beta\|_{L^{s'_{\max}}_{t,x}}$,
 $\|F\|_{L^2_{t,x}}$, $\|u\|_{L^{\infty}_t L^2_x}$ and $\|u_0\|_{L^{2}_x}$. Finally, the following local energy equality holds: for any $\psi\in C^\infty_c(\Omega)$ and for a.e. $t \in (0,T)$
\begin{equation}\label{eq:local_energy_equality}
\begin{split}
    &\frac{1}{2}\int_\Omega |u(t, x) + \nabla p_h(t, x)|^2 \, \psi(x) \diff x + \int_{0}^t \int_\Omega  (\alpha + \theta \beta):D(\psi(x)(u + \nabla p_h)(\tau,x)) \diff x \diff \tau = \\
    &= \frac{1}{2}\int_{\Omega}  |u_0(x)|^2 \, \psi(x) \diff x + \int_{0}^t \int_\Omega  (u\otimes u) : \nabla (\psi (u + \nabla p_h))\diff x\diff \tau + \\
    &\phantom{=} +  \int_{0}^t\int_{\Omega}F(u + \nabla p_h)\psi\diff x \diff \tau - \int_{0}^t \int_\Omega  \left(\sum_{i=1}^N (p_1^i + p_2^i) + p_3 + p_4\right)\,(u + \nabla p_h )\nabla\psi \diff x \diff \tau.
\end{split}
\end{equation}
\end{thm}

\section{The approximating problem}\label{section:5}

We begin our analysis with an approximation of the stress tensor $S$, which is already a staple of the theory (see eg. \cite{bulicek2021parabolic, bulicek2023nonnewtonian} and other). We set
\begin{align}\label{eq:def_of_reg_operator}
    S^\theta(t, x, \xi) := S(t, x, \xi) + \theta\nabla_{\xi}m(|\xi|), \qquad  m(|\xi|) := |\xi|^{s_{\max}}.
\end{align}
The goal this setting is an easier description of our system whenever $S^\theta$ is substituted for $S$. It turns out that $S^\theta$ exhibits growth and coercivity in classical Lebesgue spaces, hence the analysis is easier and one can use already established results. We will state it more clearly in Lemma \ref{lem:prop_of_Stheta}. As mentioned, our approximating problem is the one with $S^\theta$ substituting $S$. Thus, we will analyze

\begin{equation}\label{eq:approx_problem}
\begin{split}
\p_t f^\theta + v\cdot\nabla_xf^\theta + \DIV_v((u^\theta - v)f^\theta) &= 0\\
\p_t(u^{\theta} + \nabla p_h^\theta) + \DIV_x (u^{\theta} \otimes u^{\theta}) - \DIV_x S^{\theta}(t,x, Du^{\theta})&=\\
=- \int_{\R^d}(u^\theta - v)f^\theta\diff v  + \sum_{i = 1}^N \nabla_x(p^{i,\theta}_1 &+ p^{i,\theta}_2) + \nabla_x(p^\theta_3 + p^{\theta}_4),\\
\DIV_x u^{\theta}&=0.
\end{split}
\end{equation}

Let us state the existence result for \eqref{eq:approx_problem}.

\begin{thm}\label{thm:existence_approximation}
Let $S$ satisfy Assumption \ref{ass:stress_tensor} and $S^\theta$ be defined in \eqref{eq:def_of_reg_operator}. Then, for any initial condition $f_0\in L^\infty(\Omega\times\R^d)$, $(1 + |v|^2)f_0\in L^1(\Omega\times\R^d)$, $u_0\in L^2_{0,\DIV}(\Omega)$, there exists a function $u^\theta\in L^\infty_t L^2_x \cap L^{s_{\max}}_t W^{1,{s_{\max}}}_{0, x}$, $Du^\theta\in L^{s_{\max}}_{t,x}$, such that $S^{\theta}(t,x,Du^{\theta}) \in L^{s'_{\max}}_{t,x}$ and $f^\theta\in L^\infty_t (L^1\cap L^\infty_{x,v})$, $f^\theta\geq 0$, for which there holds
\begin{align}\label{app:existence_kinetic}
    \int_{0}^T\int_{\Omega\times\R^d}f^\theta(\p_t\phi + v\cdot\nabla_x\phi + (u^\theta - v)\cdot\nabla_v\phi)\diff v\diff x\diff t = \int_{\Omega\times\R^d}f_0\,\phi(0, x, v)\diff v\diff x
\end{align}
for any $\phi\in C^\infty_c([0, T)\times \overline{\Omega}\times \R^d)$, as well as
\begin{equation}\label{app:existence}
\begin{split}
    &\int_{\Omega_T}-u^\theta \cdot \p_t \phi - u^\theta\otimes u^\theta : \nabla\phi + S^\theta(t, x, Du^\theta):D\phi\diff x\diff t\\
    &\qquad = -\int_{0}^T\int_{\Omega}\int_{\R^d}(u^\theta - v)f^\theta \cdot \phi\diff v\diff x\diff t + \int_{\Omega}u_0(x) \cdot\phi(0,x)\diff x
\end{split}
\end{equation}
for any vector-valued $\phi\in C^\infty_c([0, T) \times \Omega)$ fulfilling $\DIV_x \phi = 0$. Moreover, the following global energy (in)equalities are satisfied for a.e. $t\in (0, T)$ (for the definition of $c^\theta, h^\theta$ see \ref{coercitivity_stress_tensor_theta_r} in Lemma \ref{lem:prop_of_Stheta})
\begin{equation}\label{app:energy}
\begin{split}
    &\frac{1}{2}\int_{\Omega}|u^\theta(t, x)|^2\diff x - \frac{1}{2}\int_{\Omega}|u_0(x)|^2\diff x + \int_0^t\int_{\Omega}S^\theta(\tau, x, Du^\theta):Du^\theta\diff x\diff \tau\\
    &\qquad = -\int_{0}^t\int_{\Omega}\int_{\R^d}(u^\theta - v)f^\theta \cdot u^\theta\diff v\diff x\diff \tau,
\end{split}
\end{equation}
\begin{equation}\label{app:energy_kinetic}
\begin{split}
    &\frac{1}{2}\int_{\Omega}|u^\theta(t, x)|^2\diff x + \frac{1}{2}\int_{\Omega\times\R^d}|v|^2f^\theta\diff v\diff x\\
    &\qquad+ \frac{1}{c^\theta}\int_0^t\int_{\Omega}|S^\theta(\tau, x, Du^\theta)|^{s'_{\max}} + |Du^\theta|^{s_{\max}}\diff x\diff \tau + \int_0^t\int_{\Omega\times\R^d}|u^\theta - v|^2f^\theta\diff v\diff x\diff\tau\\
    &\leq \frac{1}{2}\int_{\Omega}|u_0(x)|^2\diff x + \frac{1}{2}\int_{\Omega\times\R^d}|v|^2f_0\diff v\diff x + \frac{1}{c^\theta}\int_0^t\int_{\Omega}h^\theta(t, x)\diff x\diff\tau.
\end{split}
\end{equation}
\end{thm}

\begin{rem}\label{rem:form_of_h_theta}
    From the forms of $c^\theta, h^\theta$ in Lemma \ref{lem:prop_of_Stheta} it will be clear that $\frac{1}{c^\theta}h^\theta$ is independent of $\theta > 0$, hence the inequality \eqref{app:energy_kinetic} provides estimates uniform in $\theta >0$.
\end{rem}

To perform the proof of the Theorem \ref{thm:existence_approximation} we first need to look at the announced behavior of $S^\theta$. It is quantified in the lemma below.

\begin{lem}\label{lem:prop_of_Stheta}
Function $S^\theta$ satisfies the following:
\begin{enumerate}[label=(R\arabic*)]
\item \label{monotonicity_theta} $S^{\theta}(t,x, \xi)$ is a Carath\'{e}odory function and $S(t, x, 0) = 0$,
\item \label{coercitivity_stress_tensor_theta} (coercitivity and growth in $L^{s(t,x)}$) there exists a positive constant $c$ and a non-negative, integrable function $h(t, x)$, such that for any $\xi\in\symm$ and almost every $(t, x) \in (0, T) \times \Omega$
$$
c \, S^{\theta}(t, x, \xi) : \xi \geq | \xi|^{s(t, x)} + | S(t, x, \xi)|^{s'(t, x)} + \theta\nabla_{\xi}m(|\xi|) \cdot \xi - h(t, x);
$$
the constant $c$ and function $h$ can be chosen independently of $\theta$,
\item \label{coercitivity_stress_tensor_theta_r} (coercivity and growth in $L^{s_{\max}}$) there exists a positive constant $c^{\theta}$ and a non-negative, integrable function $h^{\theta}(t, x)$, such that for any $\xi\in\symm$ and almost every $(t, x) \in (0, T) \times \Omega$
$$
c^{\theta} \, S^{\theta}(t, x, \xi) : \xi \geq | \xi|^{s_{\max}} + | S^{\theta}(t, x, \xi)|^{s'_{\max}} - h^\theta(t, x)
$$
\item \label{monotonicity_stress_tensor_theta}(monotonicity) S is strictily monotone, i. e.:
$$
(S^{\theta}(t, x, \xi_1) - S^{\theta}(t, x, \xi_2)) : (\xi_1 - \xi_2) > 0
$$
for all $\xi_1 \neq \xi_2\in \symm$ and almost every $(t, x) \in (0, T) \times \Omega$.
\end{enumerate}
\end{lem}
\begin{proof}
The whole proof can be found in \cite[Lemma 5.3]{bulicek2023nonnewtonian}. We only note here that
$$
\frac{1}{c^\theta}h^\theta(t, x) = \frac{1}{c}(h(t, x) + 1),
$$
where $c >0$ and $h$ are taken from \ref{coercitivity_stress_tensor_theta}.
\end{proof}

\begin{proof}[Proof of Theorem \ref{thm:existence_approximation}] Since we have already established that $S^\theta$ behaves as in a standard theory with classical Lebesgue spaces, we can follow already known results. The proof of this theorem follows the lines of the proof in \cite{choi2025coupledvlasovnonnewtonianfluid}, which is done in a periodic domain. In fact, our case is much simpler, since in a bounded domain there is no need to split the fluid velocity into its mean-less part, and its mean (this is the main problem in a periodic domain, because the system \eqref{sys:main_system} does not have, a typical in a periodic theory, fluid velocity with zero mean over the periodic torus). Moreover, by Assumption \ref{ass:usual_bounds_exp} we know $s_{\mathrm{max}}\geq \frac{11}{5}$, hence the solution is a valid test function (as the term $(u\cdot\nabla_x)u\cdot u\in L^1_{t,x}$), and we do not need to utilize the $L^\infty$ truncation method developed in \cite{MR1713880} to identify the limits of the non-linearities. Instead, one could easily use the standard theory from \cite{gwiazda2008onnonnewtonian}. There are two questions that we need to answer. First is about the existence of regular solutions of the kinetic equation with reflection boundary conditions, which is the first step in any approximation procedure. This is an already classical theory, we refer to \cite{yu2013global, hamdache1998global, beals1987abstract}. The second is the validity \eqref{app:energy}, which is not provided in \cite{choi2025coupledvlasovnonnewtonianfluid}. To do this, notice that the function
$$
F^\theta(t, x) := -\int_{\R^d}(u^\theta - v)f^\theta \diff v,
$$
belongs to $L^2_{t,x}$ by Young's inequality, \eqref{app:energy_kinetic} and the fact that $f^\theta\in L^1_{t, x, v}$. Hence, one may use the standard theory of non-Newtonian fluids found in \cite[Theorem 5.1]{bulicek2023nonnewtonian} or \cite{gwiazda2008onnonnewtonian}, where the energy equality holds whenever the forcing term is at least in $L^1_t L^2_x$.
\end{proof}

 Obviously, we will want now to let $\theta \to 0$ in \eqref{eq:approx_problem}, \eqref{app:existence_kinetic} and \eqref{app:existence}, and complete our proof. First, we show that particular estimates independent of $\theta$ are true.

\begin{thm}\label{res:estimate_on_approx_seq}
Let $\{(f^\theta, u^{\theta})\}_{\theta \in (0,1)}$ be the sequence of solutions to \eqref{eq:approx_problem} constructed in Theorem \ref{thm:existence_approximation}. Let $\{p^{i,\theta}_1\}_{\theta \in (0,1)}$, $\{p^{i,\theta}_2\}_{\theta \in (0,1)}$, $\{p^\theta_3\}_{\theta \in (0,1)}$, $\{p^{\theta}_4\}_{\theta \in (0,1)}$, $\{p^{\theta}_h\}_{\theta \in (0,1)}$ be the sequences of pressures obtained by Theorem \ref{thm:local_energy_equality} with $
\alpha = S(t,x,Du^{\theta})$, $\beta = \nabla_{\xi}m(|D u^{\theta}|)
$. Then,
\begin{enumerate}[label = (B\arabic*)]
    \item \label{lemmathet_kinetic_est_1} $\{f^\theta\}_{\theta\in (0, 1)}$ is bounded in $L^\infty_tL^c_{x,v}$, $1\leq c \leq +\infty$,
    \item \label{lemmathet_kinetic_est_2} $\left\{\int_{\R^d}f^\theta\diff v\right\}_{\theta \in (0,1)}$ is bounded in $L^\infty_tL^c_x$, $1\leq c < \frac{d + 2}{d}$,
    \item \label{lemmathet_kinetic_est_3} $\left\{\int_{\R^d}vf^\theta\diff v\right\}_{\theta\in (0, 1)}$ is bounded in $L^\infty_t L^c_x$, $1 \leq c < \frac{d + 2}{d + 1}$.
\end{enumerate}
as well as
\begin{enumerate}[label=(B\arabic*), start = 4]
    \item\label{lemmathet_est1} $\{u^{\theta}\}_{\theta \in (0,1)}$ is bounded in $L^{\infty}_t L^2_x$,
    \item\label{lemmathet_est2} $\{Du^{\theta}\}_{\theta \in (0,1)}$ is bounded in $L^{s(t,x)}_{t,x}$,
   \item\label{lemmathet_est3i1/2} $\{ u^{\theta}\}_{\theta \in (0,1)}$ is bounded in $L^{s_{\min}}_t W^{1,s_{\min}}_{x,0}$ and $L^{q_i(t)}(0,T; W^{1,q_i(t)}(\mathcal{B}^i_{2r}))$,
    \item\label{lemmathet_est2and1/2} $\{u^{\theta}\}_{\theta \in (0,1)}$ is bounded in $L^{s_0}_{t,x}$ and $L^{R_i(t)}((0,T)\times \mathcal{B}^i_{2r})$,
    \item\label{lemmathet_est3}  $\{S(t,x,Du^{\theta})\}_{\theta \in (0,1)}$ is bounded in $L^{s'(t,x)}_{t,x}$,
    \item\label{lemmathet_est4} $\{\theta \, |D u^{\theta}|^{s_{\max}}\}_{\theta \in (0,1)}$ is bounded in $L^1_{t,x}$,
    \item\label{lemmathet_est5} $\{\theta^{1-s'_{\max}} \left|\theta\,\nabla_{\xi}m(|D u^{\theta}|)\right|^{s'_{\max}}\}_{\theta \in (0,1)}$ is bounded in $L^1_{t,x}$,
\item\label{lemmathet_est4and1/2-A} $\{p^{i,\theta}_1\}_{\theta\in(0, 1)}$ is bounded in $L^{r_i'(t)}_{t,x}$ and $L^{s'_{\max}}(0,T; L^{\infty}_{\text{loc}}(\R^d\setminus \mathcal{B}^i_{r}))$,
\item\label{lemmathet_est4and1/2-C} $\{p^{i,\theta}_2\}_{\theta\in(0, 1)}$ is bounded in $L^{R_i(t)/2}_{t,x}$ and $L^{s_0/2}(0,T; L^{\infty}_{\text{loc}}(\R^d\setminus \mathcal{B}^i_{r}))$,
\item\label{lemmatheta_est_p3} $\{p_3^\theta\}_{\theta\in (0, 1)}$ is bounded in $L^2_{t,x}$
\item\label{lemmathet_est4and1/2-F} $\{ \theta^{-1/s_{\max}} \, p^\theta_4\}_{\theta\in(0, 1)}$ is bounded in $L^{s'_{\max}}_{t,x}$
    \item\label{lemmathet_est9and5/8} $\{p^\theta_h\}_{\theta\in(0, 1)}$ is bounded in $L^\infty_t L^{s'_{\max}}_x$
    \item\label{lemmathet_est6} $\{p^{\theta}_h\}_{\theta\in (0, 1)}$ is bounded in $L^{\infty}_tW^{2,\infty}_{x, loc}$,
    \item\label{lemmathet_est8} $\{\partial_t u^\theta\}_{\theta\in (0, 1)}$ is bounded in $L^1_t V_{2,d}^*$,
    \item\label{lemmathet_est7} $\{ \partial_t (u^{\theta} + \nabla p_h^\theta) \}_{\theta \in (0,1)}$ is bounded in $L^{1}_t\left( W^{1,s_{\max}}_{x,0}\right)^*$.
    %\item\label{lemmathet_est10_before} $\{\p_t (\nabla p^\theta_h)\}_{\theta\in (0, 1)}$ is bounded in $L^1_t V_{2,b}^*$,
    %\item\label{lemmathet_est10} $\{\p_t D(\nabla p^\theta_h)\}_{\theta\in (0, 1)}$ is bounded in $L^1_t V_{3,b}^*$,
\end{enumerate}
where $V_{2,d}$ is closure of $\{\phi\in C^\infty_c(\Omega) | \DIV\phi = 0\}$ in $W^{2,d}(\Omega)$.
\end{thm}

\begin{proof}[Proof of Theorem \ref{res:estimate_on_approx_seq}]
We begin with the bounds on $f^\theta$. The first one follows from the approximation procedure in Theorem \ref{thm:existence_approximation}. Looking at \cite[Lemma 2.1]{choi2025coupledvlasovnonnewtonianfluid}, we see
$$
\sup_{t\in (0, T)}\|f^\theta\|_{L^c_{x,v}} \leq e^{d\frac{c-1}{c}T}\|f_0\|_{L^c_{x,v}},\quad 1\leq c\leq +\infty,
$$
which is enough for \ref{lemmathet_kinetic_est_1}, since by interpolation $f_0\in L^c_{x,v}$ for any $1\leq c\leq +\infty$. To get \ref{lemmathet_kinetic_est_2} and \ref{lemmathet_kinetic_est_3} one can follow \cite[Lemma 2.4]{karper2013existence} or look at \cite[Proposition 2.1]{choi2025coupledvlasovnonnewtonianfluid} to see that
\begin{align*}
    &\left|\int_{\R^d}f^\theta\diff v\right|^c \leq C_1\|f^\theta\|_\infty^{c-1}\int_{\R^d}(1+|v|^2)f^\theta\diff v,\quad 1\leq c <\frac{d+2}{d}\\
    &\left|\int_{\R^d}vf^\theta\diff v\right|^c \leq C_2\|f^\theta\|_\infty^{c-1}\int_{\R^d}(1+|v|^2)f^\theta\diff v,\quad 1\leq c <\frac{d+2}{d+1},
\end{align*}
where the right-hand side is bounded by \eqref{app:energy_kinetic}, and \ref{lemmathet_kinetic_est_1}.

 We move over to the bounds on velocity. In fact, the argumentation for the velocity field is the same as in \cite[Theorem 5.2]{bulicek2023nonnewtonian}, but let us repeat the important points for the convenience of the reader. For simplicity, let us denote for the rest of the argument
\begin{align*}
    F^\theta(t, x) := -\int_{\R^d}(u^\theta - v)f^\theta \diff v,
\end{align*}
which by \eqref{app:energy_kinetic} and \ref{lemmathet_kinetic_est_1} is uniformly bounded in $L^2_{t,x}$. We begin by noticing that \ref{lemmathet_est1} is a direct consequence of \eqref{app:energy_kinetic}, and the fact that $\frac{1}{c^\theta}h^\theta$ is independent of $\theta >0$ (see Remark \ref{rem:form_of_h_theta}). Then, we may combine the energy equality~\eqref{app:energy} and coercivity estimate~\ref{coercitivity_stress_tensor_theta} in~Lemma~\ref{lem:prop_of_Stheta} to deduce
\begin{multline*}
   \frac{1}{2}\int_{\Omega}|u^\theta(t, x)|^2\diff x + \frac{1}{c} \int_{\Omega_t} \left(
   | Du|^{s(\tau, x)}
   +  | S(\tau, x, Du)|^{s'(\tau, x)}
   + c\,\theta\nabla_{\xi}m(|D u^{\theta}|) \cdot D u^{\theta} \right) \diff x \diff \tau = \\ =  \frac{1}{2}\int_{\Omega}|u_0(x)|^2\diff x  + \int_0^t \int_{\Omega}F^\theta\cdot u^\theta\diff x \diff \tau +  \int_0^t \int_{\Omega} h(\tau,x) \diff x \diff \tau.
 \end{multline*}
By the H\"{o}lder inequality, we can estimate $\int_{\Omega} F^\theta\cdot u^{\theta}\le \|F^\theta\|_2\|u^{\theta}\|_2$ on the right hand side.  Using \ref{lemmathet_est1} and the bound on $F^\theta$, we conclude \ref{lemmathet_est2} and \ref{lemmathet_est3}. Moreover, it follows, that $\int_{\Omega_T} \theta\nabla_{\xi}m(|D u^{\theta}|) \cdot D u^{\theta}  \diff x \diff \tau$ is bounded uniformly in $\theta \in (0,1)$. But then, using the definition of a convex conjugate, we may deduce that there exists a constant $C_*>0$, such that
 $$
 \int_{\Omega_T} \left(\theta \, |Du|^{s_{\max}} + \theta \, C_*\,|\nabla_\xi m(|Du|)|^{s'_{\max}}\right) \diff x \diff \tau = \int_{\Omega_T} \theta\nabla_{\xi}m(|D u^{\theta}|) \cdot D u^{\theta}  \diff x \diff \tau.
 $$
 This implies \ref{lemmathet_est4} and \ref{lemmathet_est5}.\\

 %I put this now for red to see what is now necessary and how to adapt it
 Moving further, to show \ref{lemmathet_est3i1/2}, we observe that \ref{lemmathet_est2} implies that $\{Du^{\theta}\}_{\theta \in (0,1)}$ is bounded in $L^{s_{\min}}_{t,x}$ and $L^{q_i(t)}_{t,x}((0,T)\times \mathcal{B}^i_{2r})$ (because $s_{\min} \leq s(t,x)$ on $\Omega_T$ and $q_i(t) \leq s(t,x)$ on $(0,T)\times \mathcal{B}^i_{2r}$). Then, Korn's inequality implies that $\{\nabla u^{\theta}\}_{\theta \in (0,1)}$ is bounded in $L^{s_{\min}}_{t,x}$ and $L^{q_i(t)}_{t,x}((0,T)\times \mathcal{B}^i_{2r})$. To end the argument, note that the velocity mean $\int_{\Omega}u^{\theta}(t,x)\diff x$ is controlled in $L^{\infty}_t$ by \ref{lemmathet_est1}, so that the claim follows by the Poincar\'{e} inequality. With this, one can see that the estimate \ref{lemmathet_est2and1/2} follows from \ref{lemmathet_est1} and \ref{lemmathet_est3i1/2} together with Lemma~\ref{thm:interpolation}.\\

 The estimates on the pressures are a direct application of the Theorem \ref{thm:local_energy_equality} and the Lemma \ref{lem:integrability_of_pi} with
$$
\alpha = S^{\theta}(t,x,Du^{\theta}), \qquad \beta = \nabla_{\xi}m(|D u^{\theta}|)
$$
so that \ref{lemmathet_est4and1/2-A}--\ref{lemmathet_est6} follows from the bound on $F^\theta$ in $L^2_{t,x},$ \ref{lemmathet_est2and1/2}, \ref{lemmathet_est3} and \ref{lemmathet_est5}.\\

 To show the bounds \ref{lemmathet_est8}, \ref{lemmathet_est7} we use the equations satisfied by the velocity with or without the pressures. For \ref{lemmathet_est8}, we have
$$
\p_t u^\theta = \underbrace{- \DIV (u^{\theta} \otimes u^{\theta}) + \DIV S^{\theta}(t,x, Du^{\theta})  + F^\theta}_{:=\, A^{\theta}}
$$
We want to prove that $A^\theta$ defines a functional on $L^\infty_t V_{2,d}$. This is clear because functions in $L^\infty_t V_{2,d}$ have spatial derivatives in $L^{\infty}_t L^{z}_x$ for all $z<\infty$ and all the functions $u^{\theta} \otimes u^{\theta}$, $S^{\theta}(t,x, Du^{\theta})$ and $F^\theta$ are at least uniformly integrable in time and belong to some $L^a$ with $a>1$ with the norm independent of $\theta \in (0,1)$. For \ref{lemmathet_est7} note that
$$
\p_t(u^{\theta} + \nabla p_h^\theta) =\underbrace{- \DIV (u^{\theta} \otimes u^{\theta}) + \DIV S^{\theta}(t,x, Du^{\theta})  + F^\theta + \,\sum_{i=1}^N\nabla (p^{i,\theta}_1 + p^{i,\theta}_2) + \nabla (p^\theta_3 + p^{\theta}_4)}_{:=\, B^\theta}.
$$
We observe that all of the functions $u^{\theta} \otimes u^{\theta}$, $S^{\theta}(t,x, Du^{\theta})$, $F^\theta$, $p_i^{\theta}$ are uniformly bounded at least in $L^1_tL^{s'_{\max}}_{x}$ (this uses inequalities $s'_{\max} \leq s'_{\min}$ and equality $s_0/2 = s'_{\min}$) so that $B^\theta$, and in consequence $\partial_t(u^{\theta}+ \nabla p_h^{\theta})$ is bounded in $L^{1}_t (W^{1,s'_{\max}}_{x_0})^*$, hence \ref{lemmathet_est7} holds.
\end{proof}

\section{Proof of existence result via the monotonicity method}\label{section:6}
\begin{proof}[Proof of Theorem \ref{thm:the_main_result}] The technique of the proof will follow the one presented in \cite{bulicek2023nonnewtonian}. We begin our analysis with the following convergence lemma, which is a consequence of the bounds established in the Theorem \ref{res:estimate_on_approx_seq}.
\begin{lem}\label{lem:convergence_in_theta_approx}
Let $(f^\theta, u^{\theta})$ be a solution to \eqref{eq:approx_problem} constructed in Theorem \ref{thm:existence_approximation}. Let $p^{i,\theta}_1$, $p^{i,\theta}_2$, $p^\theta_3$, $p^{\theta}_4$, $p^\theta_h$ be the sequences of pressures obtained in Theorem \ref{thm:local_energy_equality}. Then, we can extract appropriate subsequences such that
\begin{enumerate}[label = (C\arabic*)]
    \item \label{conv_f_theta} $f^\theta \wstar f$ in $L^\infty_tL^c_{x,v}$, $1 < c\leq +\infty$,
    \item \label{conv_rho_f_theta} $\int_{\R^d}f^\theta\diff v \wstar \int_{\R^d}f\diff v$ in $L^\infty_t L^c_x$, $1 < c < \frac{d + 2}{d}$,
    \item \label{conv_momentum_f_theta} $\int_{\R^d}vf^\theta\diff v \wstar \int_{\R^d}vf\diff v$ in $L^\infty_t L^c_x$, $1 < c < \frac{d+2}{d+1}$,
    \item \label{conv_rho_f_theta_strong} $\int_{\R^d}f^\theta\diff v \rightarrow \int_{\R^d}f\diff v$ a.e. in $\Omega_T$, and in $L^{c_1}_t L^{c_2}_x$, $1\leq c_1 < +\infty$, $1\leq c_2 < \frac{d + 2}{d}$,
    \item \label{conv_momentum_f_theta_strong} $\int_{\R^d}vf^\theta\diff v \rightarrow \int_{\R^d}vf\diff v$ a.e. in $\Omega_T$, and in $L^{c_1}_t L^{c_2}_x$, $1\leq c_1 < +\infty$, $1\leq c_2 < \frac{d+2}{d+1}$,
\end{enumerate}
as well as
\begin{enumerate}[label=(C\arabic*), start = 6]                        
\item\label{c0} $u^{\theta}\overset{*}{\rightharpoonup} u$ in $L^{\infty}_{t}L^2_x$,
\item\label{conv_item_strongu} $u^{\theta} \to u$ a.e. in $\Omega_T$ and in $L^{c}_{t,x}$ for all $c < s_0$,
    \item\label{conv_item_strongu_c1c2} $u^{\theta} \to u$ in $L^{c_1}_{t} L^{c_2}_{x}$ for all $c_1 < \infty$ and $c_2<2$,
    \item\label{conv_item_strongu_local} $u^{\theta} \to u$ in $L^{R_i(t) - \delta}((0,T)\times \mathcal{B}^i_{2r})$ for all $\delta>0$,
    \
    \item\label{conv_item_sing_u} $\theta^{1/s_{\max}}\, u^{\theta} \to 0$ in $L^{s_{\max}}_{t,x}$,
    \item\label{conv_item_weak_monotone_operator_S} $S(t,x,Du^{\theta}) \rightharpoonup \chi$ in $L^{s'(t,x)}(\Omega_T)$ and $L^{r_i'(t)}_{t,x}((0,T)\times \mathcal{B}^i_{2r})$ for some $\chi \in L^{s'(t,x)}(\Omega_T)$,
    \item\label{conv_item_Dutheta} $Du^{\theta} \rightharpoonup Du$ weakly in $L^{s(t,x)}(\Omega_T)$,
    \item\label{conv_item_L1_weak} $\theta\,\nabla_{\xi}m(|D u^{\theta}|) \to 0$ in $L^{s_{\max}}_{t,x}$,
\item\label{conv_item_est4and1/2-A} $p^{i,\theta}_1 \overset{*}{\rightharpoonup} \widetilde{p^i_1}$ in $L^{r_i'(t)}_{t,x}$ and $L^{s'_{\max}}(0,T; L^{\infty}_{\text{loc}}(\R^d\setminus \mathcal{B}^i_{r}))$,
\item\label{conv_item_est4and1/2-C} $p^{i,\theta}_2 \overset{*}{\rightharpoonup} \widetilde{p^i_2}$ in $L^{R_i(t)/2}_{t,x}$ and $L^{s_0/2}(0,T; L^{\infty}_{\text{loc}}(\R^d\setminus \mathcal{B}^i_{r}))$,
\item \label{conv_item_p3} $p_3^\theta \rightharpoonup \widetilde{p_3}$ in $L^2_{t,x}$
\item\label{conv_item_est4and1/2-F} $p^\theta_4 \to 0$ in $L^{s'_{\max}}_{t,x}$,
    \item\label{conv_item_est9and5/8} $p^\theta_h \overset{*}{\rightharpoonup} \widetilde{p_h}$ in $L^{s'_{\max}}_{t,x}$ and $L^{\infty}_tW^{2,\infty}_{x, loc}$,
    \item\label{conv_item_est_add_1} $\nabla p^\theta_h \to \nabla \widetilde{p_h}$ in $L^{c}_{t} L^c_{x,loc}$ for all $c <\infty$,
    \item\label{conv_item_est_add_2/5} $\nabla^2 p^\theta_h \to \nabla^2 \widetilde{p_h}$ in $L^{c}_{t} L^c_{x,loc}$ for all $c <\infty$,
    \item\label{conv_item_est_add_2} $D(\nabla p^\theta_h) \to D(\nabla \widetilde{p_h})$ in $L^{c}_{t} L^c_{x,loc}$ for all $c <\infty$,
    \item\label{conv_item_spaceintegrals} $\int_{\Omega} |(u^{\theta} + \nabla p_h^\theta)(t,x)|^2\psi(x) \diff x \to \int_{\Omega} |(u + \nabla \widetilde{p_h})(t,x)|^2\psi(x) \diff x$ for a.e. $t \in [0,T]$ and $\psi \in C^\infty_c(\Omega)$,
    \item\label{conv_theta_forcing_term} $\int_{\R^d}(u^\theta - v)f^\theta\diff v \rightharpoonup \int_{\R^d}(u - v)f\diff v$ in $L^2_{t,x}$.
\end{enumerate}
\end{lem}

\begin{proof}
 We begin with the explanation for ${f^\theta}$. The convergences \ref{conv_f_theta}, \ref{conv_rho_f_theta}, \ref{conv_momentum_f_theta} follow from the bounds \ref{lemmathet_kinetic_est_1}, \ref{lemmathet_kinetic_est_2}, \ref{lemmathet_kinetic_est_3} and the Banach--Alaoglu theorem. The only question is the identification of the limits in \ref{conv_rho_f_theta} and \ref{conv_momentum_f_theta}. We will do that soon, but first we employ the celebrated Velocity averaging lemma \ref{velocity_averaging} (see also \cite[Lemma 5.1]{rein2004globalweak} or \cite{perthrame1998alimiting}) to deduce the strong compactness of $\{\int_{\R^d}f^\theta\diff v\}$ and $\{\int_{\R^d}v f^\theta\diff v\}$ in $L^1_{t,x}$. Hence, by interpolation we get \ref{conv_rho_f_theta_strong}, \ref{conv_momentum_f_theta_strong}, but again we need to identify the limit. Now, we have everything we need to manage it. Let
$$
\int_{\R^d}f^\theta\diff v \rightarrow \Lambda,\text{ in }L^1_{t,x}.
$$
Then, for any $\phi\in C^\infty_c([0, T]\times\Omega)$
\begin{align}\label{conv:strong_to_Lambda}
\int_{\Omega_T}\int_{\R^d}\phi\,f^\theta\diff v\diff x\diff t \rightarrow \int_{\Omega_T}\phi\,\Lambda\diff x\diff t.
\end{align}
Fix $\varphi_k\in C^\infty_c(\R^d)$, such that $\varphi_k\equiv 1$, for $|\varphi_k|> k+1$, $\varphi_k\equiv 0$ for $|\varphi_k| \leq k$, $0\leq\varphi_k\leq 1$. Then, 
\begin{align*}
    &\left|\int_{\Omega_T}\int_{\R^d}\phi\,f^\theta\diff v\diff x\diff t - \int_{\Omega_T}\int_{\R^d}\phi\,f\diff v\diff x\diff t\right|\leq \left|\int_{\Omega_T}\int_{\R^d}\phi\,(1-\varphi_k)f^\theta\diff v\diff x\diff t\right|\\
    &\qquad+ \left|\int_{\Omega_T}\int_{\R^d}\phi\,\varphi_k(f^\theta - f)\diff v\diff x\diff t\right| + \left|\int_{\Omega_T}\int_{\R^d}\phi\,(1-\varphi_k)f\diff v\diff x\diff t\right|\\
    & =: I_1^k + I_2^k + I_3^k.
\end{align*}
By \eqref{app:energy_kinetic}
$$
I_1^k \leq \|\phi\|_\infty\int_{\Omega_T}\int_{\{|v| > k\}}f^\theta\diff v\diff x\diff t \leq \frac{\|\phi\|_\infty}{k^2}\int_{\Omega_T}\int_{\R^d}|v|^2f^\theta\diff v\diff x\diff t \leq \frac{C}{k^2}.
$$
Using \ref{conv_f_theta}
$$
\int_{\Omega_T}\int_{\R^d}\phi\,\varphi_k(v)f^\theta\diff v\diff x\diff t \rightarrow \int_{\Omega_T}\int_{\R^d}\phi\,\varphi_k(v)f\diff v\diff x\diff t,
$$
thus $I_2^k \rightarrow 0$. At last, due to the Lebesgue's dominated convergence theorem $I_3^k\rightarrow 0$. Hence,
\begin{align*}
    \int_{\Omega_T}\int_{\R^d}\phi\,f^\theta\diff v\diff x\diff t \rightarrow \int_{\Omega_T}\int_{\R^d}\phi\,f\diff v\diff x\diff t.
\end{align*}
Comparing it to \eqref{conv:strong_to_Lambda} we get the needed
$$
\Lambda = \int_{\R^d}f\diff v,\text{ for a.e. }(t, x)\in \Omega_T.
$$
Analogously, one can identify the limit for $\{\int_{\R^d}vf^\theta\diff v\}$.

 Moving to the convergences of the velocity field, we again note that the arguments can be found in \cite{bulicek2023nonnewtonian}, but we provide them for the sake of completeness. First, note that using \ref{lemmathet_est3i1/2} and \ref{lemmathet_est8}, the sequence $\{u^{\theta}\}$ is strongly compact in $L^1_{t,x}$ by Aubin--Lions lemma \ref{aubin-lions}. Hence, the strong convergences \ref{conv_item_strongu}, \ref{conv_item_strongu_c1c2} and \ref{conv_item_strongu_local} follow from the interpolation and the bounds in $L^{s_0}_{t,x}$ (see \ref{lemmathet_est2and1/2}), $L^\infty_t L^2_x$ and $L^{R_i(t)}((0,T)\times \mathcal{B}^i_{2r})$ respectively. To see \ref{conv_item_sing_u}, we note that \ref{lemmathet_est4} and Korn's inequality implies uniform bound $\{\theta^{1/s_{\max}} \nabla u^{\theta}\}$ in $L^{s_{\max}}_{t,x}$ so that by Sobolev embedding and Dirichlet boundary condition, we have uniform bound $\{\theta^{1/s_{\max}} u^{\theta}\}$ in $L^{c}_{t,x}$ for some $c > s_{\max}$. As $\theta^{1/s_{\max}} u^{\theta} \to 0$ in $L^1_{t,x}$, we conclude by interpolation. Moving further, the convergence results \ref{conv_item_weak_monotone_operator_S}--\ref{conv_item_est9and5/8}
%\ref{conv_item_weak_monotone_operator_S}, \ref{conv_item_Dutheta}, \ref{conv_item_L1_weak} \ref{conv_item_est4and1/2-A}, \ref{conv_item_est4and1/2-C}, \ref{conv_item_est4and1/2-F}
follow from the Banach--Alaoglu theorem and estimates \ref{lemmathet_est3}, \ref{lemmathet_est2} and \ref{lemmathet_est5}--\ref{lemmathet_est6},
%\ref{lemmathet_est5},  \ref{lemmathet_est4and1/2-A},  \ref{lemmathet_est4and1/2-C}, \ref{lemmathet_est4and1/2-F}
respectively. Next, we can use \ref{lemmathet_est3i1/2}, \ref{lemmathet_est6},  \ref{lemmathet_est7} and the Aubin--Lions \ref{aubin-lions} lemma to conclude that 
$$
(u^{\theta}+\nabla p_h^{\theta}) \to (u+\nabla \widetilde{p_h}) \qquad \textrm{ in } L^{1}_{t} L^{1}_{x, loc}.
$$
Thus, \ref{conv_item_est_add_1} follows from \ref{lemmathet_est6}. Finally, since $p_h^{\theta}$ is harmonic with respect to the spatial variable, we have that $\|p^{\theta}_h - \widetilde{p_h}\|_{W^{k,2}(\Omega'')}\le C(k,\Omega'', \Omega')\|p^{\theta}_h - \widetilde{p_h}\|_{L^1(\Omega')}$ for all $\Omega'' \Subset \Omega' \subset \Omega$ and all $k$. Consequently, \ref{conv_item_est_add_2/5} and \ref{conv_item_est_add_2} follow from \ref{conv_item_est_add_1}. The property~\ref{conv_item_spaceintegrals} holds true because of the presence of the function $\psi$ having compact support in $\Omega$ and thus we can combine \ref{conv_item_est_add_1} and \ref{conv_item_strongu_c1c2} and use the classical properties of the Lebesgue spaces. At last for \ref{conv_theta_forcing_term} we note that \ref{conv_rho_f_theta_strong}, \ref{conv_momentum_f_theta_strong}, \ref{conv_item_strongu} imply
$$
\int_{\R^d}(u^\theta - v)f^\theta\diff v \rightarrow \int_{\R^d}(u - v)f\diff v,\text{ a.e. in }\Omega_T.
$$
Since by \eqref{app:energy_kinetic}, \ref{lemmathet_kinetic_est_1} $\{\int_{\R^d}(u^\theta - v)f^\theta\diff v\}$ is bounded in $L^2_{t,x}$, we get what is needed by the Banach--Alaoglu theorem.
\end{proof}

 Now, we recall our standard notation
$$
F^\theta := -\int_{\R^d}(u^\theta - v)f^\theta\diff v,\quad F := -\int_{\R^d}(u - v)f\diff v.
$$
Then, for each $\theta \in (0,1)$ we use Theorem~\ref{thm:existence_approximation} to obtain the distributional formulations of kinetic equation and momentum equations without pressure
\begin{align}\label{app:existence_final_proof_2}
    \int_{0}^T\int_{\Omega\times\R^d}f^\theta(\p_t\phi + v\cdot\nabla_x\phi + (u^\theta - v)\cdot\nabla_v\phi)\diff v\diff x\diff t = \int_{\Omega\times\R^d}f_0\,\phi(0, x, v)\diff v\diff x,
\end{align}
for all $\phi\in C^\infty_c([0, T)\times\overline{\Omega}\times\R^d)$,
\begin{align}\label{app:existence_in_final_proof}
    \int_{\Omega_T}-u^\theta \cdot \p_t\phi - u^\theta\otimes u^\theta : \nabla\phi + S^\theta(t, x, Du^\theta):D\phi\diff x\diff t = \int_{\Omega_T}F^\theta \cdot \phi\diff x\diff t + \int_{\Omega}u_0(x)\cdot \phi(0,x)\diff x,
\end{align}
satisfied for all vector-valued  $\phi\in C^\infty_c([0, T) \times \Omega)$ fulfilling  $\DIV \phi = 0$.  We can let $\theta \to 0$ in \eqref{app:existence_final_proof_2}, \eqref{app:existence_in_final_proof} to obtain
\begin{align}\label{app:existence_in_final_proof_after_limit_2}
    \int_{0}^T\int_{\Omega\times\R^d}f(\p_t\phi + v\cdot\nabla_x\phi + (u - v)\cdot\nabla_v\phi)\diff v\diff x\diff t = \int_{\Omega\times\R^d}f_0\,\phi(0, x, v)\diff v\diff x,
\end{align}
and
\begin{align}\label{app:existence_in_final_proof_after_limit}
    \int_{\Omega_T}-u \cdot \p_t \phi- u \otimes u : \nabla\phi + \chi:D\phi\diff x\diff t = \int_{\Omega_T}F \cdot\phi\diff x\diff t + \int_{\Omega}u_0(x) \cdot\phi(0,x)\diff x.
\end{align}
The convergences that need more explanation are the ones with the terms $f^\theta(u^\theta - v)\cdot\nabla_v\phi$, and the operator $S^{\theta}(t,x,Du^{\theta})$. For the first term, we note that $v\cdot\nabla_v\phi\in C^\infty_c([0, T)\times\Omega\times\R^d)$, which by \ref{conv_f_theta} implies
$$
\int_{\Omega_T}\int_{\R^d}f^\theta\,v\cdot\nabla_v\phi\diff v\diff x\diff t \rightarrow \int_{\Omega_T}\int_{\R^d}f\,v\cdot\nabla_v\phi\diff v\diff x\diff t.
$$
Since $\nabla_v\phi$ has a compact support in the kinetic velocity variable, by \ref{conv_item_strongu} $u^\theta\rightarrow u$ in $L^3((0, T)\times\Omega\times\supp\nabla_v\phi)$. Combining it with \ref{conv_f_theta} we obtain
$$
\int_{\Omega_T}\int_{\R^d}f^\theta\,u^\theta\cdot\nabla_v\phi\diff v\diff x\diff t \rightarrow \int_{\Omega_T}\int_{\R^d}f\,u\cdot\nabla_v\phi\diff v\diff x\diff t.
$$
As for the stress tensor, by \eqref{eq:def_of_reg_operator}, we may write $S^{\theta}(t,x,Du^{\theta}) = S(t,x,Du^{\theta}) + \theta \nabla_{\xi} m(|Du^{\theta}|)$. Then, by \ref{conv_item_L1_weak}, we know that the regularizing term converges in $L^1_t L^1_x$, which is sufficient to perform the desired passage to the limit $\theta \to 0$.\\

 Having already \eqref{app:existence_in_final_proof_after_limit_2} - \eqref{app:existence_in_final_proof_after_limit}, our proof will conclude after we show $\chi(t,x) = S(t,x,Du)$. Our goal will be to utilize the classical Minty's monotonicity trick. Hence, we will need to use some form of energy equality for our limiting equation, which will allow us to derive a needed bound on the mixed term
 $$
   \limsup_{\theta \to 0}\int_0^t\int_\Omega S(t, x, Du^\theta): Du^\theta \psi(x)\diff x\diff \tau.
 $$
 Since such an equality is not available for the momentum equation without pressure \eqref{app:existence_in_final_proof_after_limit}, this is the exact place where we need the provided Theorem \ref{thm:local_energy_equality}. Indeed, let us apply Theorem~\ref{thm:local_energy_equality}. First, we look at the available formulation with pressures
\begin{equation}\label{eq:weak_form_with_press_and_IC_in_final_proof}
\begin{split}
    \int_{\Omega_T}-&(u^{\theta} + \nabla p_h^\theta)\cdot \p_t\phi - u^{\theta}\otimes u^{\theta}:\nabla\phi \,+\, S^{\theta}(t,x,Du^{\theta}) : D\phi\diff x \diff t = \\ 
    = \int_{\Omega}&u_0(x)\cdot \phi(0, x) \diff x -  \int_{\Omega_T} \left(\sum_{i = 1}^N(p_1^{i,\theta}+p_2^{i,\theta} ) + p^\theta_3 + p_4^{\theta} \right)\,\DIV\phi\diff x \diff t + \int_{\Omega_T}F^\theta \cdot \phi\diff x \diff t
\end{split}
\end{equation}
satisfied for all $\phi\in C^\infty_c([0, T) \times \Omega)$. We can let $\theta \to 0$ in \eqref{eq:weak_form_with_press_and_IC_in_final_proof} similarly as above to obtain
\begin{equation}\label{eq:weak_form_with_press_and_IC_in_final_proof_after_limit}
\begin{split}
    \int_{\Omega_T}-&(u + \nabla \widetilde{p_h}) \cdot \p_t\phi - u\otimes u:\nabla\phi \,+\, \chi : D\phi\diff x \diff t = \\
     = \int_{\Omega}&u_0(x) \cdot\phi(0, x) \diff x - \left(\int_{\Omega_T}\sum_{i = 1}^N (\widetilde{p^i_1} +\widetilde{p^i_2}) + \widetilde{p_3} + \widetilde{p_4}\right) \,\DIV\phi\diff x \diff t + \int_{\Omega_T}F \cdot \phi\diff x \diff t.
\end{split}
\end{equation}
To have a point of comparison, we can apply Theorem~\ref{thm:local_energy_equality} directly to \eqref{app:existence_in_final_proof_after_limit}. This yields pressures $p^i_1$, $p^i_2$, $p_3$, $p_4$ and $p_h$ with a distributional formulation as \eqref{eq:weak_form_with_press_and_IC_in_final_proof_after_limit} but with $\widetilde{p_j^i}$, $\widetilde{p_j}$ and $\widetilde{p_h}$ replaced by ${p_j^i}$, $p_j$ and $p_h$ respectively. We would like to identify the respective pressures to obtain the same distributional formulations. On one hand, using simply the uniqueness (linearity) established in the Lemma \ref{lem:integrability_of_pi}, we can easily deduce $\widetilde{p_j^i} = p_j^i$ and $\widetilde{p_j} = p_j$ almost everywhere. On the other hand, $p_h$ is obtained uniquely up to the condition
$$
\int_{\Omega}p_h(t, x)\diff x = 0,
$$
hence it is enough to show that
$$
\int_{\Omega}\widetilde{p_h}(t, x)\diff x = 0.
$$
Indeed, from the weak convergence \ref{conv_item_est9and5/8}, the strong convergence \ref{conv_item_est_add_1} and the Poincar\'{e} inequality, we may deduce that for almost all $t\in (0,T)$
$$
0 = \lim_{\theta\to 0}\int_{\Omega}p_h^\theta(t, x)\diff x = \int_{\Omega}\widetilde{p_h}(t, x)\diff x.
$$
The distributional formulations \eqref{eq:weak_form_with_press_and_IC_in_final_proof} and \eqref{eq:weak_form_with_press_and_IC_in_final_proof_after_limit} are not the only things we obtain through the Theorem \ref{thm:local_energy_equality}. More importantly, we can deduce the local energy equalities obtained from \eqref{app:existence_in_final_proof} and \eqref{app:existence_in_final_proof_after_limit}. We have respectively
\begin{equation}\label{local_energy_equality_theta}
\begin{split}
    \frac{1}{2}&\int_\Omega |u^{\theta}(t, x) + \nabla p_h^{\theta}(t, x)|^2 \, \psi(x) \diff x + \int_{0}^t \int_\Omega  S^{\theta}(\tau, x, Du^\theta):D(\psi(x)(u^{\theta}+\nabla p_h^\theta)(\tau,x)) \diff x \diff \tau = \\
    &= \frac{1}{2}\int_\Omega  |u_0(x)|^2 \, \psi(x) \diff x + \int_{0}^t \int_\Omega (u^{\theta}\otimes u^{\theta}) : \nabla (\psi (u^{\theta} + \nabla p_h^{\theta}))\diff x\diff \tau + \\
    &+  \int_{0}^t\int_{\Omega}F^\theta\cdot(u^{\theta} + \nabla p_h^{\theta})\psi \diff x \diff \tau - \int_{0}^t \int_{\Omega} \left(\sum_{i=1}^N (p_1^{i,\theta} + p_2^{i,\theta}) + p^\theta_3 + p_4^{\theta} \right)(u^{\theta} + \nabla p_h^{\theta} )\cdot\nabla\psi \diff x \diff \tau
\end{split}
\end{equation}
and also 
\begin{equation}\label{local_energy_equality_after_limit_pass}
\begin{split}
    \frac{1}{2}&\int_\Omega |u(t, x) + \nabla p_h(t, x)|^2 \, \psi(x) \diff x + \int_{0}^t \int_\Omega  \chi(\tau, x):D(\psi(x)(u + \nabla p_h)(\tau,x)) \diff x \diff \tau = \\
    &= \frac{1}{2}\int_{\Omega}  |u_0(x)|^2 \, \psi(x) \diff x + \int_{0}^t \int_\Omega  (u\otimes u) : \nabla (\psi (u + \nabla p_h))\diff x\diff \tau + \\
    &+  \int_{0}^t\int_{\Omega}F\cdot(u + \nabla p_h)\psi\diff x \diff \tau -  \int_{0}^t \int_{\Omega} \left(\sum_{i=1}^N (p_1^i + p_2^i) + p_3  \right)(u + \nabla p_h )\cdot \nabla\psi \diff x \diff \tau,
\end{split}
\end{equation}
for a.e. $t \in (0,T)$ and all test functions $\psi \in C_c^{\infty}(\Omega)$. Note that in principle \eqref{local_energy_equality_after_limit_pass} should be stated with $\widetilde{p_j^i}$, $\widetilde{p_j}$ and $\widetilde{p_h}$ replacing ${p_j^i}$, $p_j$ and $p_h$, but we were able to identify them almost everywhere in the previous consideration. As mentioned at the beginning, we want now to compare the equality \eqref{local_energy_equality_theta} with \eqref{local_energy_equality_after_limit_pass} in the limit $\theta \to 0$ to identify $\chi$ via monotonicity arguments. \\

Let us start with the simple terms, whose limits are easily obtainable via the Lemma \ref{lem:convergence_in_theta_approx}. We notice that a direct application of \ref{conv_item_spaceintegrals} yields
\begin{align}\label{eq:conv_energy_with_theta_1}
    \frac{1}{2}\int_\Omega |u^{\theta}(t, x) + \nabla p_h^{\theta}(t, x)|^2\psi(x)\diff x \to \frac{1}{2}\int_\Omega |u(t, x) + \nabla p_h(t, x)|^2\psi(x)\diff x,
\end{align}
and by \ref{conv_theta_forcing_term} together with  \ref{conv_item_est_add_1},  \ref{conv_item_strongu} we know
\begin{equation}\label{eq:conv_energy_with_theta_2}
    \int_{0}^t \int_{\Omega} F^\theta\cdot (u^\theta + \nabla p_h^{\theta}) \psi \diff x \diff \tau \to \int_{0}^t \int_{\Omega}  F \cdot(u + \nabla p_h)\psi \diff x \diff \tau,
\end{equation}
moreover \ref{conv_item_p3} with \ref{conv_item_est_add_1}, \ref{conv_item_strongu} provides us
\begin{equation}\label{eq:convergence_term_p_3_theta}
\int_{0}^t \int_{\Omega} p^\theta_3(u^{\theta} + \nabla p_h^{\theta} )\cdot\nabla\psi \diff x \diff \tau\to \int_{0}^t \int_{\Omega} p_3(u + \nabla p_h )\cdot\nabla\psi \diff x \diff \tau,
\end{equation}
where all of the convergences are for a.e. $t\in (0, T)$. Furthermore, using the estimate \ref{lemmathet_est4and1/2-F} and convergences \ref{conv_item_sing_u} and \ref{conv_item_est_add_1}, we deduce
$$
\left|\int_{0}^t \int_{\Omega} p_4^{\theta}(u^{\theta} + \nabla p_h^{\theta} )\nabla\psi \diff x \diff \tau\right| \leq \|p_4^{\theta}\, \theta^{-1/s_{\max}}\|_{L^{s'_{\max}}_{t,x}} \, \|\theta^{1/s_{\max}}(u^{\theta} + \nabla p_h^{\theta} )\nabla\psi \|_{L^{s_{\max}}_{t,x}} \to 0.
$$
Thus,
\begin{equation}\label{eq:convergence_term_p_4_theta}
\int_{0}^t \int_{\Omega} p_4^{\theta}(u^{\theta} + \nabla p_h^{\theta} )\cdot\nabla\psi \diff x \diff \tau\to 0.
\end{equation}
The one term that needs a longer explanation is the convergence 
\begin{equation}\label{eq:conv_energy_with_theta_3}
\int_{0}^t \int_\Omega (u^{\theta}\otimes u^{\theta}) : \nabla (\psi (u^{\theta} + \nabla p_h^\theta))\diff x\diff \tau \to \int_{0}^t \int_\Omega (u\otimes u) : \nabla (\psi (u + \nabla p_h))\diff x\diff \tau.
\end{equation}
We split the argument into two parts using the equality $\psi (u^{\theta} + \nabla p_h^\theta) = \psi u^{\theta} + \psi \nabla p_h^\theta$.

For the term $\int_{0}^t \int_\Omega (u^{\theta}\otimes u^{\theta}) : \nabla (\psi \nabla p_h^{\theta})\diff x\diff \tau$ the convergence of the integral is a simple consequence of \ref{conv_item_est_add_1}, \ref{conv_item_est_add_2/5} and $u^{\theta} \to u$ in $L^2_{t,x}$ from \ref{conv_item_strongu}.

For the term $\int_{0}^t \int_\Omega (u^{\theta}\otimes u^{\theta}) : \nabla (\psi u^{\theta})\diff x\diff \tau$ we have the following identity
$$
\int_{0}^t \int_\Omega (u^{\theta}\otimes u^{\theta}) : \nabla (\psi u^{\theta}) \diff x\diff \tau=  \frac{1}{2}
\int_{0}^t \int_\Omega |u^{\theta}|^2\, u^{\theta} \cdot \nabla \psi \diff x\diff \tau - \frac{1}{2}\int_{0}^t \int_\Omega \psi \DIV u^{\theta}\,  |u^{\theta}|^2 \diff x\diff \tau.
$$
The first term converges to $\frac{1}{2}
\int_{0}^t \int_\Omega |u|^2\, u \cdot \nabla \psi \diff x\diff \tau$ because $u^{\theta} \to u$ strongly in $L^3_{t,x}$ as in \ref{conv_item_strongu} (note that $s_0>3$). The second term vanishes by the incompressibility condition so that we obtain \eqref{eq:conv_energy_with_theta_3}.\\

Let us move over to the troublesome terms. To conclude our argument, we need to show that for $i=1,...,N$ and $j = 1, 2$ the convergences
\begin{equation}\label{eq:conv_press_p_1_p_2_theta}
\int_{0}^t \int_{\Omega} p_j^{i,\theta}(u^{\theta} + \nabla p_h^{\theta} )\cdot\nabla\psi \diff x \diff \tau\to \int_{0}^t \int_{\Omega} p_j^i(u + \nabla p_h )\cdot\nabla\psi \diff x \diff \tau
\end{equation}
hold. We separate those terms from other considerations, as in this place we shall utilize the different, complementary regularity results for the fluid velocity field and the pressures; split the integral into integrals over $\Omega \cap \mathcal{B}^{i}_{2r}$ and over $\Omega \setminus \mathcal{B}^{i}_{2r}$. Let $j=1$ and $i \in \{1,...,N\}$ be fixed. 

When looking at the integral over $\Omega\cap\mathcal{B}^{i}_{2r}$, we use that $p_1^{i,\theta}$ converges weakly in $L^{r_i'(t)}_{t,x}$, so it is sufficient to have strong convergence of $u^{\theta} + \nabla p_h^{\theta}$ in $L^{r_i(t)}_{t,x,loc}$ thanks to the compact support of $\psi$. This follows from \ref{conv_item_strongu_local} and \ref{conv_item_est_add_1} as $R_i(t)-r_i(t)\geq \frac{s_{\min}}{d}$.

On the other hand, on $\Omega \setminus \mathcal{B}^{i}_{2r}$ we use the weak$^*$ convergence of $p_1^{i,\theta}$ in $L^{s'_{\max}}_t L^{\infty}_x$ from \ref{conv_item_est4and1/2-A} and local strong convergence of $u^{\theta} + \nabla p_h^{\theta}$ in $L^{s_{\max}}_t L^1_x$ from \ref{conv_item_strongu_c1c2} and \ref{conv_item_est_add_1}.

Now, let $j = 2$ and $i \in \{1,...,N\}$ be fixed. As above, we look at the integral over $\Omega \cap \mathcal{B}^{i}_{2r}$ and over $\Omega \setminus \mathcal{B}^{i}_{2r}$.
 
When looking at the integral over $\Omega\cap\mathcal{B}^{i}_{2r}$ we use that $p_2^{i,\theta}$ converges weakly in $L^{R_i(t)/2}_{t,x}$, so it is sufficient to have strong convergence of $u^{\theta} + \nabla p_h^{\theta}$ in $L^{(R_i(t)/2)'}_{t,x,loc}$. However, we have $\left(\frac{R_i(t)}{2} \right)' < R_i(t)$  because $R_i(t) > 3$ (note that we already checked this in Section~\ref{section:4} below \eqref{eq:regularity_estimates_not_u_2}). Therefore, the required strong convergence follows from \ref{conv_item_strongu_local} and \ref{conv_item_est_add_1}.

On the other hand, on $\Omega \setminus \mathcal{B}^{i}_{2r}$ we use the weak$^*$ convergence of $p_2^{i,\theta}$ in $L^{s_0/2}_t L^{\infty}_x$ from \ref{conv_item_est4and1/2-C} and local strong convergence of $u^{\theta} + \nabla p_h^{\theta}$ in $L^{s_0/2}_t L^1_x$ from \ref{conv_item_strongu_c1c2} and \ref{conv_item_est_add_1}.\\

Combining the proven convergences \eqref{eq:conv_energy_with_theta_1}--\eqref{eq:conv_press_p_1_p_2_theta}, we obtain the following preliminary inequality 
\begin{equation}\label{eq:summary_step_convergences_theta}
\begin{split}
\limsup_{\theta \to 0} \int_{0}^t \int_\Omega  S^{\theta}(\tau, x, Du^\theta):D(\psi(x)&(u^{\theta}+\nabla p_h^\theta)(\tau,x)) \diff x \diff \tau\\
&\leq \int_{0}^t \int_\Omega  \chi(\tau, x):D(\psi(x)(u + \nabla p_h)(\tau,x)) \diff x \diff \tau,
\end{split}
\end{equation}
true for a.e. $t\in (0, T)$. We call it preliminary, because as mentioned before our goal is to provide a proof, that for a.e. $t \in (0,T)$ and $\psi \in C_c^{\infty}(\Omega)$
\begin{align}\label{conv:monotonicity_trick_0}
    \limsup_{\theta \to 0}\int_0^t\int_\Omega S(t, x, Du^\theta): Du^\theta \psi(x)\diff x\diff \tau \leq \int_0^t \int_\Omega \chi(t, x) : Du\,\psi(x)\diff x\diff \tau
\end{align}
is true. Indeed, let us argue that the preliminary inequality \eqref{eq:summary_step_convergences_theta} implies the standard inequality \eqref{conv:monotonicity_trick_0}. To this end, we decompose term on the (LHS) of \eqref{eq:summary_step_convergences_theta} into six parts $I_1$, $I_2$, $I_3$, $I_4$, $I_5$, $I_6$ as follows
\begin{align*}
& \int_{0}^t \int_\Omega  S^{\theta}(\tau, x, Du^\theta):D(\psi(x)(u^{\theta} + \nabla p_h^\theta)(\tau,x)) \diff x \diff \tau = \\
&  = \int_{0}^t \int_\Omega S(\tau, x, Du^\theta):Du^{\theta} \psi(x)  \diff x \diff \tau + \int_{0}^t \int_\Omega S(\tau, x, Du^\theta):[D(\nabla p_h^\theta)\,\psi(x)] \diff x \diff \tau\\
& \phantom{ = } + \int_{0}^t \int_\Omega S(\tau, x, Du^\theta):[\nabla\psi(x) \otimes (u^{\theta} + \nabla p_h^\theta)] \diff x \diff \tau + \int_{0}^t \int_\Omega \theta \, \nabla_\xi m(|Du^{\theta}|):D u^{\theta} \, \psi(x) \diff x\diff \tau \\
& \phantom{ = } + \int_{0}^t \int_\Omega \theta \, \nabla_\xi m(|Du^{\theta}|):D (\nabla p_h^{\theta}) \, \psi(x) \diff x\diff \tau + \int_{0}^t \int_\Omega \theta \, \nabla_\xi m(|Du^{\theta}|):(\nabla \psi(x) \otimes (u^{\theta} + \nabla p_h^\theta))  \diff x \diff \tau\\
&=: I_1 + I_2 + I_3 + I_4 + I_5 + I_6.
\end{align*}
The integral $I_1$ is the one appearing in \eqref{conv:monotonicity_trick_0}, and this is the one that we want to estimate. Moving further, looking at $I_2$, we get
\begin{equation}\label{eq:conv_theta_term_X2}
    \int_0^t \int_\Omega S(\tau, x, Du^\theta) : [D(\nabla p_h^\theta) \psi(x)]\diff x\diff \tau \to \int_0^t \int_\Omega \chi(\tau,x) : [D(\nabla p_h)\psi(x)]\diff x\diff \tau,
\end{equation}
since $S(t,x,Du^{\theta}) \rightharpoonup \chi$ in $L^{s'(t,x)}$ by \ref{conv_item_weak_monotone_operator_S}, so that $S(t,x,Du^{\theta}) \rightharpoonup \chi$ in $L^{s'_{\max}}_{t,x}$ and $D(\nabla p_h^\theta) \to D(\nabla p_h)$ in $L^{s_{\max}}_{t,x, loc}$, by \eqref{conv_item_est_add_2}. Next, for $I_3$, we will argue that the convergence
\begin{equation}\label{eq:conv_theta_term_X3}
\int_0^t \int_\Omega S(\tau, x, Du^\theta) : [\nabla\psi \otimes (u^\theta + \nabla p_h^\theta)]\diff x\diff \tau \to \int_0^t \int_\Omega \chi(\tau,x) : [\nabla \psi\otimes (u + \nabla p_h)]\diff x\diff \tau,
\end{equation}
holds. Again, we shall use the special properties of our partition of the domain, and the complementary convergence results. Let $1 = \sum_{i=1}^N \zeta_i$, where $\{\zeta_i\}$ is the partition of unity from Notation \ref{not:zeta}. For fixed $i\in \{1, ..., N\}$ we focus on the convergence of the integral
$$
\int_0^t \int_\Omega \zeta_i\, S(\tau, x, Du^\theta) : [\nabla\psi\otimes (u^\theta + \nabla p_h^\theta)]\diff x\diff \tau.
$$
Since $\zeta_i$ is supported in $\mathcal{B}^i_r$ we can use weak convergence of $S(\tau, x, Du^\theta)$ in $L^{r_i'(t)}_{t,x}$ from \ref{conv_item_weak_monotone_operator_S} and the strong convergence $u^{\theta} + \nabla p_h^\theta \to u + \nabla p_h$ in $L^{r_i(t)}_{t,x,loc}$ from \ref{conv_item_est_add_1} and \ref{conv_item_strongu_local} (this uses also $R_i(t)-r_i(t)\geq \frac{s_{\min}}{d}$). Summing over $i$, we get what is needed. Furthermore, for the terms $I_4$, $I_5$, $I_6$ we get
\begin{equation}\label{eq:X4_nonnegativity_and_X5}
I_4 \geq 0, \qquad \qquad I_5 \to 0, \qquad \qquad I_6 \to 0.
\end{equation}
Indeed, the convergence $I_5 \to 0$ follows directly from \ref{conv_item_L1_weak} and estimate \ref{lemmathet_est6}. Moreover for $I_6$, the argument is analogous to the one for \eqref{eq:convergence_term_p_4_theta}, because we have exactly the same integrability of $\theta \, \nabla_\xi m(|Du^{\theta}|)$ as of $p_4^{\theta}$. Plugging \eqref{eq:conv_theta_term_X2}--\eqref{eq:X4_nonnegativity_and_X5} into \eqref{eq:summary_step_convergences_theta} we obtain \eqref{conv:monotonicity_trick_0}. At last, we are able to end our proof via the monotonicity trick argument. Using the Assumption \ref{monotonicity_stress_tensor} we get
\begin{align}\label{eq:monoton_inequality_with_theta_and_eta}
    \int_{\Omega_T}(S(t, x, Du^\theta) - S(t, x, \eta)) : (Du^\theta - \eta) \, \psi(x) \geq 0
\end{align}
for any $\eta\in L^\infty_t L^\infty_x$. Now, let us study limits of two terms appearing in \eqref{eq:monoton_inequality_with_theta_and_eta}. First, since $\eta\in L^\infty_t L^{\infty}_x$, then $S(t, x, \eta)\in L^\infty_t L^{\infty}_x$. Hence, using the weak convergence \ref{conv_item_Dutheta}, we obtain
\begin{align}\label{conv:monotonicity_trick_1}
    \int_{\Omega_T}S(t, x, \eta) :  Du^\theta \, \psi(x) \diff x\diff t\to \int_{\Omega_T}S(t, x, \eta) : Du \,\psi(x)\diff x\diff t.
\end{align}
From \ref{conv_item_weak_monotone_operator_S} we easily deduce
\begin{align}\label{conv:monotonicity_trick_2}
    \int_{\Omega_T}S(t, x, Du^\theta) :  \eta \, \psi(x)\diff x\diff t \to \int_{\Omega_T}\chi(t, x) : \eta \, \psi(x)\diff x\diff t.
\end{align}
Thus, applying \eqref{conv:monotonicity_trick_0}, \eqref{conv:monotonicity_trick_1} and \eqref{conv:monotonicity_trick_2}, we may take $\limsup_{\theta \to 0}$ in \eqref{eq:monoton_inequality_with_theta_and_eta} to get
\begin{align*}
    \int_{\Omega_T}(\chi(t, x) - S(t, x, \eta)) : (Du - \eta) \, \psi(x)\diff x\diff t \geq 0.
\end{align*}
Using Lemma \ref{res:monot_trick} (Minty's monotonicity trick), we finally obtain $\chi(t,x) = S(t,x,Du)$ a.e..
\end{proof}

\appendix

\section{Musielak-Orlicz spaces \texorpdfstring{$L^{s(t, x)}(\Omega_T)$}{Ls}}\label{app:musielaki}

 Here we mark some of the basic properties of the variable exponent spaces $L^{s(t, x)}(\Omega_T)$. For the more details on the Musielak-Orlicz spaces, see \cite{chlebicka2019book}. We start with the definition

\begin{Def}\label{def:musielak_exponent_space}
Given a measurable function $s(t,x): \Omega_T \to [1,\infty)$, we let
\begin{align*}
L^{s(t,x)}(\Omega_T) = \left\{ \xi: \Omega_T \to \mathbb{R}^d: \mbox{ there is } \lambda>0 \mbox{ such that } \int_{\Omega_T} \left| \frac{\xi(t,x)}{\lambda}\right|^{s(t,x)} \diff x \diff t < \infty \right\}.
\end{align*}
Moreover if $s(t, x)$ satisfies the boundedness conditions \ref{ass:usual_bounds_exp} then this definition is equivalent to
$$
L^{s(t,x)}(\Omega_T) = \left\{ \xi: \Omega_T \to \mathbb{R}^d:  \int_{\Omega_T} \left| {\xi(t,x)}\right|^{s(t,x)} \diff x \diff t < \infty \right\}.
$$
\end{Def}

 Variable exponent spaces are Banach with the following norm

\begin{thm}
Let
\begin{equation}\label{intro:norm}
\| \xi \|_{L^{s(t,x)}} = \inf \left\{\lambda>0: \int_{\Omega_T} \left| \frac{\xi(t,x)}{\lambda}\right|^{s(t,x)} \diff x \diff t \leq 1 \right\}.
\end{equation}
Then, $\left(L^{s(t, x)}, \|\cdot\|_{L^{s(t, x)}}\right)$ is a Banach space.
\end{thm}

 We will be interested in the two kinds of convergences in the aformentioned spaces

\begin{Def}
We say that $\xi_n$ converges strongly to $\xi$ (denoted $\xi_n \to \xi$) in $L^{s(t, x)}(\Omega_T)$, if
$$
\|\xi_n - \xi\|_{L^{s(t, x)}}\rightarrow 0
$$
and that $\xi_n$ converges modularly to $\xi$, if there exists $\lambda > 0$ such that
$$
\int_{\Omega_T} \left| \frac{\xi_n(t,x) - \xi(t,x)}{\lambda} \right|^{s(t,x)} \diff x \diff t \to 0.
$$
\end{Def}

 Modular and strong convergences are connected in the following form

\begin{thm} The following equivalence holds true:
\begin{align*}
    \|\xi_n - \xi\|_{L^{s(t, x)}} \to 0 \Longleftrightarrow \int_{\Omega_T} \left| \frac{\xi_n(t,x) - \xi(t,x)}{\lambda} \right|^{s(t,x)} \diff x \diff t \to 0 \text{ for every }\lambda >0
\end{align*}
\end{thm}

\begin{cor}
If $s(t, x)$ satisfies the boundedness conditions \ref{ass:usual_bounds_exp}, then strong and modular convergences are equivalent.
\end{cor}

 The Theorem below makes a connection between modular convergence and the convergences of the products as formulated below.

\begin{thm}\label{thm:modular_l1}\textbf{\textup{(Proposition 2.2, \cite{gwiazda2008onnonnewtonian})}}
Assume that the function $s(t, x)$ satisfies Assumption \ref{ass:exponent_cont_space}. Presuppose also that $\phi_n \to \phi$ modularly in $L^{s(t, x)}(\Omega_T)$ and $\psi_n\to \psi$ modularly in $L^{s'(t, x)}(\Omega_T)$. Then, $\phi_n\,\psi_n \to \phi\,\psi$ in $L^1(\Omega_T)$.
\end{thm}

\section{Auxiliary propositions}

\begin{lem}\label{thm:interpolation}
Suppose that $v \in L^{\infty}_t L^2_x\cap L^q_t W^{1,q}_x$ and $q \geq 2$. Then, $v \in L^{r_0}_t L^{r_0}_x$ where $r_0 = q \left(1 + \frac{2}{d} \right)$ and
$$
\|v\|_{L^{r_0}_t L^{r_0}_x} \leq C(\|v\|_{L^{\infty}_t L^2_x}, \|v\|_{L^q_t W^{1,q}_x}).
$$
\end{lem}

\begin{lem}\label{L2convlemma}
Let $1 \leq p < \infty$ and $\{u_n\}_{n \in \N}$ be a sequence such that $u_n \to u$ in $L^p_t L^p_x$. Then, there exists a subsequence $\{u_{n_k}\}_{k \in \N}$, such that
$$
\mbox{for a.e. } t \in (0,T) \qquad u_{n_k}(t,x) \to u(t,x) \mbox{ in } L^p_x
$$
Moreover, if $\{u_{n_k}\}_{k \in \N}$ is bounded in $L^{\infty}_t L^2_x$ we have for a.e. $t \in (0,T)$
$$
\int_{\Omega} |u(t,x)|^2 \diff x \leq \liminf_{k \to \infty} \int_{\Omega} |u_{n_k}(t,x)|^2 \diff x
$$
\end{lem}
%%%%
%%%proof commented
\iffalse
\begin{proof}
Consider sequence of functions $\{f_n\}_{n \in \N}$ defined with
$$
f_n(t):= \left(\int_{\Omega}|u_n(t,x) - u(t,x)|^p \diff x\right)^{1/p}.
$$
Then, $f_n \to 0$ in $L^p_t$ so it has a subsequence $f_{n_k}$ that converges a.e. on $(0,T)$. It follows that $u_{n_k}(t,x) \to u(t,x)$ in $L^p_x$ for a.e. $t \in (0,T)$.\\

 Concerning the second statement, let $\mathcal{N}$ be a subset of $(0,T)$ consisting of times for which the latter convergence holds and let $\mathcal{M}$ be a subset of $(0,T)$ for which
$$
\|u_{n_k}(t,\cdot)\|_{L^2_x} \leq \sup_{k \in \N} \|u_{n_k}\|_{L^{\infty}_t L^2_x}.
$$
Clearly, $(0,T) \setminus (\mathcal{N} \cap \mathcal{M})$ is a null set. For $t \in \mathcal{N}$ we have $u_{n_k}(t,x) \rightharpoonup u(t,x)$ in $L^p_x$. Moreover, standard subsequence argument combined with Banach-Alaoglu shows that for $t \in \mathcal{N} \cap \mathcal{M}$ we have $u_{n_k}(t,x) \rightharpoonup u(t,x)$ in $L^2_x$. The conclusion follows by weak lower semicontinuity of the (squared) $L^2_x$ norm.
\end{proof}
\fi
%%%%%%%%%%
%%%proof commented

\begin{lem}\label{res:monot_trick}\textup{\textbf{(Lemma 2.16, \cite{bulicek2021parabolic})}}
Let $S$ be as in Assumption \ref{ass:stress_tensor}. Assume there are $\chi \in L^{s'(t,x)}(\Omega_T)$ and $\xi \in L^{s(t,x)}(\Omega_T)$, such that
\begin{equation*}
\int_{\Omega_T} \left(\chi - S(t,x,\eta)\right) : (\xi - \eta)\, \psi(x)\diff t \diff x \geq 0
\end{equation*}
for all $\eta \in L^{\infty}(\Omega_T; \R^d)$ and $\psi \in C_0^{\infty}(\Omega)$ with $0 \leq \psi \leq 1$. Then,
$$
S(t,x,\xi) = \chi(t,x) \mbox{ a.e. in } \Omega_T.
$$
\end{lem}

\begin{lem}\label{aubin-lions}\textup{\textbf{(Generalized Aubin--Lions lemma, \cite[Lemma 7.7]{MR3014456})}}
Denote by
$$
W^{1, p, q}(I; X_1, X_2) := \left\{u\in L^p(I; X_1); \frac{du}{dt}\in L^q(I; X_2)\right\}
$$
Then if $X_1$ is a separable, reflexive Banach space, $X_2$ is a Banach space and $X_3$ is a metrizable locally convex Hausdorff space, $X_1$ embeds compactly into $X_2$, $X_2$ embeds continuously into $X_3$, $1 < p <\infty$ and $1\leq q\leq \infty$, we have
$$
W^{1, p, q}(I; X_1, X_3) \text{ embeds compactly into }L^p(I; X_2)
$$
In particular any bounded sequence in $W^{1, p, q}(I; X_1, X_3)$ has a convergent subsequence in $L^p(I; X_2)$.
\end{lem}

\begin{lem}\label{velocity_averaging}\textup{\textbf{(Velocity averaging lemma \cite[Lemma 2.7]{karper2013existence})}}
Suppose $\{f^n\}$ bounded in $L^p_{loc}(\R^{2d+1})$, and $\{G^n\}$ bounded in $L^p_{loc}(\R^{2d+1})$ for $1 < p < +\infty$, be such that
$$
\p_tf^n + v\cdot\nabla_x f^n = \nabla_v^kG^n,\quad f^n(0, x, v) = f_0(x, v)\in L^p(\R^{2d + 1}).
$$
Moreover suppose that $\{f^n\}$ is bounded in $L^\infty(\R^{2d+1})$, and $\{(|x|^2 + |v|^2)f^n\}$ is bounded in $L^\infty(0, T; L^1(\R^{2d}))$. Then, for any $|\psi(v)|\leq c|v|$, and $q < \frac{d+2}{d+1}$
$$
\left\{\int_{\R^d}\psi(v)f^n\diff v\right\}
$$
is relatively compact in $L^q((0, T)\times\R^d)$.
\end{lem}

\bibliographystyle{abbrv}
\bibliography{parpde_mo_discmodul}
\end{document}